\documentclass[12pt]{amsart}
\usepackage{latexsym}
\usepackage{amsmath,amsfonts,amssymb,amsthm}
\usepackage{graphicx}
\usepackage{xcolor}
\input xy
\xyoption{all}
\newcommand{\toric}[1]{\T\V({#1})}  %
\newcommand{\dual}{^{\scriptscriptstyle\vee}}
\newcommand{\Singular}{{\sc Singular}}
\newcommand{\Normaliz}{{\sc Normaliz}}

\definecolor{tomato}{rgb}{1.0,.388,.278}      %
\definecolor{DarkSalmon}{rgb}{.914,.588,.478} %
\definecolor{SaddleBrown}{rgb}{0.545,.271,.075}      %

\newcommand{\A}{\mathbb A}
\newcommand{\C}{\mathbb C}

\newcommand{\N}{\mathbb N}

\newcommand{\Q}{\mathbb Q}
\newcommand{\R}{\mathbb R}
\newcommand{\T}{\mathbb T}
\newcommand{\V}{\mathbb V}
\newcommand{\Z}{\mathbb Z}

\newcommand{\CM}{{\mathcal M}}

\newcommand{\CO}{{\mathcal O}}

\newcommand{\CR}{{\mathcal R}}

\DeclareMathOperator{\Spec}{Spec}

\newcommand{\conv}{\operatorname{conv}} %

\DeclareMathOperator{\pr}{pr}

\newcommand{\surj}{\rightarrow\hspace{-0.8em}\rightarrow}

\newcommand{\ko}{\overline}
\newcommand{\ku}{\underline}

\newcommand{\kss}{\scriptscriptstyle}
\newcommand{\kd}{\displaystyle}

\newcommand{\keps}{\varepsilon}
\newcommand{\veee}{{\scriptscriptstyle\vee}}
\newcommand{\kbb}{{\kss \bullet}}

\newcommand{\kst}{\,|\;}

\swapnumbers
\newtheorem{theorem}[subsection]{Theorem}
\newtheorem*{theX}{Theorem}
\newtheorem{lemma}[subsection]{Lemma}

\newtheorem*{corX}{Corollary}
\newtheorem{proposition}[subsection]{Proposition}

\theoremstyle{definition}

\newtheorem*{defX}{Definition}

\newtheorem*{exaX}{Example}

\theoremstyle{remark}
\newtheorem*{remark}{Remark}

\numberwithin{equation}{section}

\definecolor{cpriv}{rgb}{0.4,0.1,0.2}

\newcommand{\kkk}[1]{}

\newcommand{\kI}{\mathcal J}
\newcommand{\ktI}{\tilde{\mathcal J}}
\newcommand{\di}{i}

\newcommand{\aj}{j}
\newcommand{\ak}{k}

\usepackage{paralist}
\usepackage{booktabs}
\usepackage{tikz}
\DeclareMathOperator{\im}{im}
\DeclareMathOperator{\interior}{int}
\DeclareMathOperator{\Cone}{Cone}
\DeclareMathOperator{\Hom}{Hom}
\DeclareMathOperator{\spa}{span}
\newcommand{\twovec}[2]{
\left(
\begin{array}{c} #1 \\ #2 \end{array}
\right)}

\newcommand{\ult}{\underline{t}} %
\newcommand{\sig}{\sigma} %
\newcommand{\sigv}{\sigma^{\veee}} %
\newcommand{\Qinf}{Q^{\infty}}
\newcommand{\Rpos}{\mathbb R_{\ge 0}}
\newcommand{\lat}{\mathbb L} %
\newcommand{\VQ}{V}
\newcommand{\m}{N}
\newcommand{\RI}{\mathcal R}
\newcommand{\ezs}{\eta_0^*}
\newcommand{\ues}{\ku{\eta}^*}
\newcommand{\uecs}{\ku{\eta}^{*c}}
\newcommand{\es}{\eta^*}
\newcommand{\taut}{\tilde{C}(Q)}
\newcommand{\tautv}{\tilde{C}(Q)^{\veee}}

\begin{document}
\title[Deformations of toric singularities]{
Negative deformations of toric singularities that are smooth 
in codimension two}
\author[K.~Altmann]{Klaus Altmann%
}
\author[L.~Kastner]{Lars Kastner}
\address{Fachbereich Mathematik und Informatik, 
Institut f\"ur Mathematik,
Freie Universit\"at Berlin,
Arnimalle 3, 
14195 Berlin, 
Germany}
\email{altmann@math.fu-berlin.de, kastner@math.fu-berlin.de
}
\subjclass[2000]{}
\keywords{}

\date{}
\begin{abstract}
Given a cone $\sigma\subseteq N_{\R}$ with smooth two-dimensional
faces and, moreover, an element $R\in\sigma\dual\cap M$ of the dual
lattice, we describe the part of the versal deformation
of the associated toric variety $\toric{\sigma}$ that 
is built from the deformation parameters of multidegree $R$.
\\
The base space is (the germ of) an affine scheme $\bar{\CM}$
that reflects certain possibilities of splitting
$Q:=\sigma\cap[R=1]$ into Minkowski summands.
\vspace{2ex}
\end{abstract}
\maketitle

\section{Introduction}\label{intro}
\subsection{}\label{intro-def}
The entire deformation theory of an isolated singularity is encoded 
in its
so-called versal deformation. For complete intersection 
singularities this
is a family over a smooth base space obtained by certain 
perturbations of the defining equations.\\
As soon as we leave this class of singularities, the structure of the family,
and sometimes even the base space, will be more complicated. It is well known
that the base space may consist of several components or may be non-reduced.
\par

\subsection{}\label{intro-toric}
Let $M,N$ be two mutually dual, free abelian groups of finite
rank. Then affine toric varieties are constructed from rational, 
polyhedral cones
$\sigma\subseteq N_\R:=N\otimes \R$: One takes the dual cone
\[
\sigv:=\{r\in M_\R|\ \langle a,r\rangle\ge 0
\mbox{ for each }a\in\sigma\},
\]
and $Y:=\toric{\sigma}$ is defined as the spectrum of the semigroup algebra
$\C[\sigv\cap M]$. In particular, the equations of $Y$
are induced from linear relations between lattice points of 
$\sigv\subseteq M_\R$. 
As usual for all other toric objects or notions, 
the toric deformation theory also comes with an $M$-grading. In
particular, for any $R\in M$, we might speak of 
infinitesimal or versal deformations of degree $-R$.
\\
With the latter, we mean the following: The 
vector space $T^1$ of infinitesimal deformations serves as the ambient space
of the germ of the versal base space. Hence it makes sense
to intersect it with the linear space obtained as the
annihilator of the $T^1$ coordinates of degree $\neq R$.
Equivalently, the versal deformation of degree $-R$
can be understood as the maximal extension of the infinitesimal
deformations in degree $-R$.

\subsection{}\label{intro-Gorenst}
For investigating versal deformation spaces, Gorenstein
singularities are the easiest examples beyond complete
intersections. It is a helpful coincidence that the 
Gorenstein property
has a very nice description in the toric context --
the cone $\sigma$ should just be spanned by a lattice polytope
$Q$ sitting in an affine hyperplane
$[R^*=1]$ of height one. 
Note that $R^*\in M$ equals the degree of the volume form.
This leads to the investigation
of the deformation theory of toric Gorenstein singularities
in \cite{gorenstein} -- the interesting deformations
were contained in degree $-R^*$.
\\
The present paper is meant as a
generalization of this approach. We discard the Gorenstein
assumption. 
For $Y$ we just assume smoothness in codimension two
(as was already done in the Gorenstein case), and for $R$
we restrict to the case of a primitive $R\in\sigma\dual\cap M$. 
Otherwise, 
one would leave the toric framework, cf.\ \cite{flip}.
\\
While the main ideas
work along the lines of \cite{gorenstein}, we try to keep
the paper as self-contained as possible.

\subsection{}\label{intro-QR}
The main tool to describe our
results is the notion of Minkowski
sums.

\begin{defX}
For two polyhedra
$P,P'\subseteq\R^n$ we define their Minkowski sum as the
polytope $P+P':=\{p+p'|\ p\in P,\ p'\in P'\}$.
Obviously, this notion also
makes sense for translation classes of polytopes.
For instance, each polyhedron $Q$ is the Minkowski sum
of a compact polytope and the so-called tail cone $Q^\infty$.
\vspace{-2ex}
\end{defX}

Let us fix a primitive
element $R$ of $\sigv\cap M$ and intersect
the cone $\sigma$ with the hyperplane defined by $[R=1]$. 
This intersection
defines a polyhedron named $Q:=Q(R)$. For our investigations, this
$Q$ plays a similar role as the $Q$ in the Gorenstein case.
However, in the present paper, 
it neither needs to be a lattice polyhedron, nor compact.
If $a^i$ is one of the primitive generators of $\sigma$,
then it leads to a lattice/non-lattice vertex of $Q$
or to a generating ray of its tail cone $Q^\infty$
iff $\langle a^i,R\rangle=1$, $\geq 2$, or $=0$, respectively.
\\
Following the Gorenstein case we will construct a
''moduli space'' $C(Q)$ of Minkowski summands of multiples of $Q$
-- but in the present paper, 
we have to take care of their possible tail cones 
as well as the non-lattice vertices of $Q$. 
Attaching each Minkowski summand at the point that represents 
it in $C(Q)$
yields the so-called tautological cone $\taut$ together with a 
projection
onto $C(Q)$. It can be seen as the universal Minkowski summand
of $Q$.
Indeed, applying the functor that makes toric varieties from cones will
provide the main step toward constructing the versal base 
space of $Y=\toric{\sigma}$ in degree $-R$.

\subsection{}\label{intro-plan}
For a given polyhedron $Q\subseteq\R^n$ we begin in Sect.~2 
by presenting an affine scheme $\bar{\CM}$.
It is related to $C(Q)$ and
describes the possibilities of splitting
$Q$ into Minkowski summands.
In Sect.\ 3 we study the tautological cone $\taut$. 
Applied to $Q(R)$,
this leads in Sect.~4 to
the construction of a flat family over $\bar{\CM}$ with $Y$ as special
fiber. 
Now we can state the main theorem (\ref{obst-mainthm}) of this paper.
\begin{theX}
The family $\bar{X}\times_{\bar{S}}\CM\to\bar{\CM}$ (cf.\ \ref{flat-them})
with base space $\bar{\mathcal M}$
is the versal deformation of $Y$ of degree $-R$,
\end{theX}
i.e.\ the Kodaira-Spencer map
is an isomorphism in degree $-R$ (Sect.\ 5) and the obstruction map
is injective (Sect.\ 6). Based on this an interesting question arises, namely
whether it is possible
to construct the part of the versal deformation of $Y$
with negative degrees by repeatedly applying the principles of this paper.\\
The last section starts with describing the situation for $\dim Y=3$
(in Theorem \ref{3dim}) and
then continues with an explicit example. 
It shows how to compute
the family using \Singular\ (cf.\ \cite{singular}) 
and \Normaliz\ (cf.\ \cite{normaliz}).

\subsection{Acknowledgement}\label{ack}
We would like to thank the 
anonymous referee for the careful reading, for checking the
calculations, and for the valuable hints.

\section{The base space}\label{def} 

\subsection{}\label{def-QR} 
Let $\sig=\langle a^1,\ldots,a^M\rangle\subseteq N_{\R}$ 
be a cone such that the two-dimensional faces
$\langle a^j,a^k\rangle <\sig$ are smooth 
(i.e.\ $a^1,\ldots,a^M\in N$ are primitive, and
$\{a^j,a^k\}$ 
could be extended to a $\Z$-basis of $N$). Let $R\in\sigv\cap M$ 
be a primitive element. Then one can define:

\begin{defX}
Let $R\in\sigma^{\veee}\cap M$ be primitive. We 
define the affine space
$\A:=[R=1]:=\{a\in N_\R\kst 
\langle a ,R\rangle=1\}\subseteq N_{\R}$ with lattice
$\lat:= \A\cap N$.
It contains the polyhedron
$Q:=Q(R):=\sigma\cap [R=1]$ with tail cone 
$\Qinf=\sigma\cap[R=0]$.
Note that $\Qinf=0$ if and only if
$R\in \interior\sigv$ .
\vspace{-2ex}
\end{defX}

Note that we can recover $\sig$ as 
$\sig=\overline{\Rpos\cdot (Q,1)}=
\Rpos\cdot (Q,1)\cup(\Qinf,0)$.
The vertices of $Q$ are $v^i=a^i/\langle a^i,R\rangle$ 
for those fundamental
generators $a^i\in\sig$ with $\langle a^i ,R\rangle\geq 1$;
they belong to $\lat$ iff $\langle a^i ,R\rangle= 1$.
We will see that $Y$ is rigid in degree $-R$ unless
$Q$ has at least one such $\lat$-vertex. Assuming this,
we fix one of the $\lat$-vertices of $Q$ to be the origin.

\subsection{}\label{def-eps} 
Denote by $d^1,\ldots, d^{\m}\in N_{\Q}$ the compact edges
of $Q$ after choosing some orientation of each of them.
Calling edges that meet in a common {\em non-lattice} vertex of $Q$
``{\em connected}'' implies that
the set $\{d^1,\ldots, d^{\m}\}$
may be uniquely decomposed into components according to this
notion.

\begin{defX}
For every compact 2-face $\keps < Q$ we can define the sign vector
$\ku{\keps}=(\keps_1,\ldots,\keps_{\m})\in\{0,\pm 1\}^{\m}$ by
\begin{equation*}
\keps_i:= \begin{cases}
  \pm 1 & \textnormal{if } d^i \textnormal{ is an edge of }\keps\\
  0 & \textnormal{otherwise} 
 \end{cases}
\end{equation*}
such that the oriented edges $\keps_i\cdot d^i$ fit into a cycle 
along the boundary of $\keps$. 
This determines $\ku{\keps}$ up to sign and we choose one of both
possibilities. 
In particular, we have $\sum_i\keps_i d^i=0$ if $\keps< Q$ 
is a compact 2-face.
\vspace{-2ex}
\end{defX}

Now we define the vector space $V(Q)\subseteq\R^{\m}$ by
\[
V(Q):=
\{ (t_1,\dots,t_N)\, |\;
\begin{array}[t]{@{}l}
\sum_i t_i  \,\keps_i  \,d^i  =0\;
\mbox{ for every compact 2-face } \keps <Q \mbox{, and}\\
t_i=t_j \; 
\mbox{if $d^i,d^j$ contain a common non-lattice vertex of }
Q\}. 
\end{array}
\]
To simplify notation we are going to use $\VQ:=V(Q)$.
For each component of edges there is a
well defined associated
coordinate function $\VQ\to\R$. 
Now, $C(Q):=\VQ\cap\R^{\m}_{\ge 0}$ is 
a rational, polyhedral cone in $\VQ$, 
and its points correspond to 
certain Minkowski summands
of positive multiples of $Q$: 

\begin{lemma} \label{def-lem}
Each point $\ku{t}\in C(Q)$ 
define a Minkowski summand of a positive multiple
of $Q$; its $i$-th compact edge equals $t_i d^i$.
This yields a bijection between $C(Q)$ and the set
of all Minkowski summands
(of positive multiples of $Q$) that change components of edges just
by a scalar.
\vspace{-2ex}
\end{lemma}

\begin{proof}
For an Element $\ult\in C(Q)$ the corresponding summand $Q_{\ult}$ is built
by the edges $t_i\cdot d^i$ as follows: 
Each vertex $v$ of $Q$ can be reached from $0\in Q$ by some walk along the compact
edges $d^i$ of $Q$. We obtain
\begin{equation*}
 v=\sum\limits_{i=1}^{\m}\lambda_id^i\textnormal{ for some }
 \ku{\lambda}=(\lambda_1,\ldots,\lambda_\m),\ \lambda_i\in\Z.
\end{equation*}
Now given an element $\ult\in C(Q)$, we may define the corresponding vertex
$v_{\ult}$ by
\begin{equation*}
 v_{\ult}:=\sum\limits_{i=1}^{\m}t_i\lambda_id^i,
\end{equation*}
and the linear equations defining 
$\VQ$ ensure that this definition does not
depend on the particular path from $v$ to $0$ through the compact part of the
$1$-skeleton of $Q$.
We define the Minkowski summand
by
$Q_{\ult}:=\conv\{v_{\ult}\} + \Qinf$.
\end{proof}

\subsection{}\label{def-higherdeg}
Now, we define a higher degree analogous to the linear equations
defining $\VQ$:

\begin{defX}
For each compact 2-face $\keps <Q$, and for each integer $k\ge 1$ we define the
vector valued polynomial
\begin{equation*}
g_{\keps,k}(\ult):=\sum_{i=1}^{\m}\,t_i^k\,\keps_i\,d^i.
\end{equation*}
\end{defX}

Using coordinates of $\A$, the $d^i$ turn into scalars,
thus the $g_{\keps,k}(\ult)$ turn into regular
polynomials; for each pair $(\keps,k)$ we
will get two linearly independent ones. Since
\[
V^\bot = \mbox{span}\,\{
\begin{array}[t]{@{}l}
\,
[\langle\keps _1 d^1,c\rangle, \dots, \langle \keps_{\m} d^{\m},c
\rangle ]\,
| \;
\keps <Q \mbox{ is a compact 2-face},\, c\in \A^*; \\
\,
[0,\dots,1_i,\dots,-1_j,\dots,0]
\,|\; \mbox{$d^i,d^j$ have a common
non-lattice $Q$-vertex}\,\},
\end{array}
\]
they (together with $t_i-t_j$ for $d^i, d^j$ sharing a common
non-lattice vertex) can be written as
\[
g_{\ku{d},k}(\ult):=\sum_{i=1}^{\m}d_it_i^k
\]
with $\ku{d}\in\VQ^{\bot}\cap\Z^{\m}$ and $k\in\N$.
We thus may define the ideal
$$
\begin{array}{rcl}
\kI &:=& (g_{\keps,k})_{\keps, k\geq 1} +
 (t_i-t_j\kst d^i, d^j
\mbox{ share a common non-lattice vertex})
\\
&=&
(g_{\ku{d},k}(\ult)|\ \ku{d}\in\VQ^{\bot}\cap\Z^{\m})
\subseteq \C[\ult]
\end{array}
$$
which defines an affine closed subscheme
\[
\CM := \mbox{Spec}\, ^{\displaystyle \C [\ult]} \!\! \left/ \!\!
_{\displaystyle \mathcal J }\right. \subseteq V_{\C} \subseteq \C^{\m}.
\]

Denote by $\ell$ the canonical projection
\[\ell:\C^{\m}\twoheadrightarrow 
 ^{\displaystyle \C^{\m}} \hspace{-0.7em} \left/ \!\!
_{\displaystyle \C\cdot (1,\ldots,1) }\right. .
\]
On the level of regular functions this corresponds 
to the inclusion
$\C[t_i-t_j|\ 1\le i,j\le \m ]\subseteq \C[\ult]$.
Note that the vector $\ku{1}=(1,\ldots,1) \in C(Q)\subseteq\VQ$ 
encodes $Q$ as a Minkowski summand of itself. 

\begin{theorem} \label{def-4} 
{\rm (1)}
$\kI $ is generated by polynomials from $\,\C [t_i -t_j ]$, i.e.\ $\CM =
\ell^{-1}(\bar{\CM})$ for the affine closed subscheme
$\bar{\CM} \subseteq
\,^{\displaystyle V_{\C}} \!\!\! \left/
\!_{\displaystyle \C\cdot \ku{1}} \right.
\subseteq
\;^{\displaystyle \C^{\m}} \!\!\! \left/
\!_{\displaystyle \C\cdot \ku{1}} \right.$
defined by $\kI  \cap \C [t_i -t_j ]$.
\\[1ex]
{\rm (2)}
$\kI \subseteq \C[t_1,\dots,t_{\m}]$ is the smallest ideal
that meets property (1) and, on the other hand, contains the 
``toric equations''
\[
\prod_{i=1}^N t_i ^{d_i^+} - \prod_{i=1}^N t_i ^{d_i^-}\quad 
\mbox{ with }
\hspace{0.4em}\ku{d}\in\VQ^{\bot}\cap\Z^{\m}.
\]
\\
(For an integer $h$ we denote
\[
h^+ := \left\{ \begin{array}{cl}
h & \mbox{ if } h \geq 0\\
0 & \mbox{ otherwise}
\end{array} \right.
\quad ; \qquad
h^- := \left\{ \begin{array}{ll}
0 & \mbox{ if } h \geq 0\\
-h & \mbox{ otherwise}
\end{array} \right. .)
\]
\end{theorem}

The proof is similar to the one of  
\cite[Theorem (2.4)]{gorenstein}.

\section{The tautological cone}\label{tautco}

\subsection{}\label{tautco-taut} 
While $C(Q)\subseteq V(Q)$ were built to describe the base space,
we turn now to the cone that will eventually lead to the total
space of our deformation.

\begin{defX}
The tautological cone $\tilde{C}(Q)\subseteq \A\times \VQ$ is defined as
\[
\tilde{C}(Q):= \; 
\{(v,\ult)\,|\; \ult\in C(Q);\, v\in Q_{\ult}\}; 
\]
it is generated by the pairs $({v}^\aj_{\ult^l},\ult^l)$ 
and $(v^\ak,0)$
where $\ult^l$, ${v}^\aj$, and
$v^\ak$ run through the generators of $C(Q)$,
vertices of $Q$, and generators of $Q^\infty$, respectively.
\end{defX}

Since $\sig=\overline{\Cone(Q)}\subseteq \A\times\R=N_\R$, 
we obtain a fiber product diagram of rational polyhedral cones:
\begin{equation*}
 \xymatrix{ [\sig\subseteq\A\times\R]\ \ar@{^(->}[r]^{i}\ar@{>>}[d]^{\pr_{\R}} 
&\  [\tilde{C}(Q)\subseteq\A\times \VQ]\ar@{>>}[d]^{\pr_{\VQ}}\\
\Rpos\ar@{^(->}[r]^{\cdot\ku{1}} & [C(Q)\subseteq \VQ]
}
\end{equation*}

The three cones $\sig\subseteq \A\times\R$,
$\tilde{C}(Q)\subseteq\A\times \VQ$ and $C(Q)\subseteq \VQ$ define affine
toric varieties called 
$Y,X$ and $S$, respectively. The corresponding rings of regular functions are
\begin{eqnarray*}
  A(Y)= &\C[\sigv\cap(\lat^*\times\Z)],\\
  A(X)= &\C[\tilde{C}(Q)^{\veee}\cap\tilde{M}],\ \tilde{M}:=\lat^*\times
  \VQ_{\Z}^*\\
A(S)= &\C [C(Q)^{\veee}\cap \VQ_{\Z}^*],
\end{eqnarray*}
and we obtain the following diagram:
\begin{equation*}
 \xymatrix{ Y\ar@{^(->}[r]^{i}\ar[d] & X\ar[d]^{\pi}\\ 
\C\ar@{^(->}[r] & S.
}
\end{equation*}
Unfortunately, this diagram does not 
need to be a fiber product diagram as we
will explain in\ (\ref{tautco-prop}).

\subsection{}\label{tautco-1}
To each non-trivial $c\in(\Qinf)^{\veee}$ we 
associate a vertex
$v(c)\in Q$ and a number $\eta_0(c)\in\R$ %
meeting the properties
\begin{eqnarray*}
\langle Q,\, c \rangle + \eta_0(c) &\geq& 0 \qquad \mbox{ and}\\
\langle v(c),\, c \rangle +\eta_0(c) &=& 0.
\end{eqnarray*}
For $c=0$ we define $v(0):=0\in\lat$ and $\eta_0(0):=0\in\R$.

\begin{remark}
(1) With respect to $Q$, $c\neq 0$ is the inner normal vector of the affine
supporting hyperplane
$[\langle \bullet,c\rangle + \eta_0(c)=0]$ through $v(c)$. In particular,
$\eta_0(c)$
is uniquely determined, while $v(c)$ is not.
\\[0.5ex]
(2)
Since $0\in Q$, the $\eta_0(c)$ are non-negative.
\vspace{-2ex}
\end{remark}

Moreover, if $c\in(\Qinf)^{\veee}\cap\lat^*$,
we denote by $\ezs(c)$ the smallest integer greater than or equal to
$\eta_0(c)$, i.e.\ $\ezs(c)=\lceil \eta_0(c)\rceil$. Then
\[
\sigv=\{[c,\eta_0(c)]\,|\;c\in (\Qinf)^{\veee}\} \;+ \;
   \R_{\geq 0}\cdot[\ku{0},1]
\]
and
\[
\sigv\cap M= \{[c,\ezs(c)]\,|\;c\in\lat^*\cap (\Qinf)^{\veee}\}
   \;+\; \N\cdot [\ku{0},1]\,.
\]
Note that $[\ku{0},1]$ equals the element $R\in M$ fixed in the 
beginning. In
particular we can choose a generating set 
$E\subseteq\sigv\cap M$  as some
\[
E=\{[\ku{0},1], [c^1,\ezs(c^1)],\dots ,[c^w,\ezs(c^w)]\}\,.
\]

\subsection{}\label{tautco-2} Thinking of $C(Q)$ as a cone in $\R^{\m}$ instead of $V$ allows
dualizing
the equation $C(Q)=\Rpos^{\m} \cap V$ to get
$C(Q)^{\veee}= \Rpos^{\m} + V^\bot$. Hence, for $C(Q)$ as a cone in $V$
we obtain
\[
C(Q)^{\veee}= \;^{\displaystyle \Rpos^{\m} + V^\bot}\!\!\left/
_{\displaystyle V^\bot} \right.
= \im\, [\Rpos^{\m} \longrightarrow V^*].
\]
The surjection $\Rpos^{\m}\surj C(Q)^{\veee}$ induces a map
$\N^{\m} \longrightarrow C(Q)^{\veee}\cap V^*_{\Z}$ which does not
need to be
surjective at all. This leads to the following definition:\\
\par

\begin{defX}
 {\em On $V^*_{\Z}$ we introduce a partial ordering ``$\ge$'' by
\[
\ku{\eta}\ge \ku{\eta}' \quad \Longleftrightarrow
\quad \ku{\eta}-\ku{\eta}' \in \mbox{\em im}\, [\N^{\m}\rightarrow
V_{\Z}^*] \subseteq C(Q)^{\veee}\cap V^*_{\Z}.
\] }
\end{defX}
On the geometric level, the non-saturated semigroup
$\im\, [\N^{\m}\rightarrow
V_{\Z}^*] \subseteq C(Q)^{\veee}\cap V^*_{\Z}$ corresponds to the
scheme theoretical
image ${\bar{S}}$ of $p:S\rightarrow \C^{\m}$, and $S\rightarrow {\bar{S}}$ is
its normalization, cf.\ (\ref{flat-pf}).
The equations of ${\bar{S}} \subseteq \C^{\m}$ are collected in the kernel of
\[
\C[t_1,\dots,t_{\m}]=\C[\N^{\m}] \stackrel{\varphi}{\longrightarrow}
\C[ C(Q)^{\veee}\cap V^*_{\Z}] \subseteq \C[V^*_{\Z}],
\]
and it is easy to see that
$$
\ker\,\varphi = \left( \left.
\prod_{i=1}^{\m} t_i ^{d_i ^+} - \prod_{i=1}^{\m} t_i ^{d_i ^-}\,
\right| \;
\ku{d}\in \Z^{\m} \cap V^\bot \right)
$$
is generated by the toric equations from (\ref{def-4}).

\subsection{}\label{tautco-6}
To deal with the dual space $\VQ^*$ the following point of 
view will be useful:
In the Gorenstein case we described its elements by using the
surjection
$\R^{\m}\to\VQ^*$. In particular, an element $\ku{\eta}\in\VQ^*$
was given by
coordinates $\eta_i$ corresponding to the edges $d^i$ of $Q$.
Now, in the general case,
the set of edges of $Q$ splits into several components,
cf.\ (\ref{def-eps}).
For each such component,                                           
not the single coordinates but only their sum 
along the entire component is well defined.
However,
this does not affect that the total summation map $\R^\m\to\R$
factors through $\VQ^*\to\R$. It will still be denoted as
$\ku{\eta}\mapsto\sum_i\eta_i$.

\begin{defX}
{\rm (1)} For 
$c\in(\Qinf)^{\veee}$
choose some path from $0\in Q$ to $v(c)\in Q$
through the 1-skeleton of $Q$ and let
$\ku{\lambda^c}:=(\lambda_1^c,\ldots,\lambda_{\m}^c)\in\Z^{\m}$ 
be the vector counting how
often (and in which direction) we went through each 
particular edge. Then 
\[ 
\ku{\eta}(c):=[-\lambda_1^c\langle d^1,c\rangle,
\ldots,-\lambda_{\m}^c\langle d^{\m},c\rangle]\in\Q^{\m} 
\]
defines an element $\ku{\eta}(c)\in \VQ^*$ not depending on the special
choice of the path $\ku{\lambda^c}$.
\\[0.5ex]
{\rm (2)}
Let $v\in Q$ be a vertex not contained in the lattice $\lat$. 
Then we denote by $e[v]\in\VQ_{\Z}^*$ the element represented by
$[0,\ldots,0,1_i,0,\ldots,0]\in\Z^{\m}$ for some compact edge $d^i$ containing
$v$, i.e.\ 
$e[v]$ yields the entry $t_i$ of $\ult\in\VQ$.
(Note that $e[v]$ does not depend on the choice of $d^i$.)
\\[0.5ex]
{\rm (3)}
For $c\in(\Qinf)^{\veee}\cap \lat^*$ denote
$\ues(c):=\ku{\eta}(c)+(\ezs(c)-\eta_0(c))\cdot e[v(c)]\in\VQ^*$.
(If $v(c)\in\lat$, then $\ezs(c)=\eta_0(c)$ implies that we do not need
$e[v(c)]$ in that case.)
\end{defX}

Here are the essential properties of $\ues(c)$:

\begin{lemma} \label{tautco-lem}
Let $c\in (Q^\infty)^{\veee}\cap \lat^*$. Then
\\[0.5ex]
{\rm (i)}
$\ues(c)\in \im [\N^{\m}\to V^*_{\Z}]
\subseteq C(Q)^{\veee}\cap V_{\Z}^*$,
and this element equals $\ku{\eta}(c)$ if and only if
$\langle v(c),c\rangle\in\Z$ (in particular, if $v(c)\in\lat$).
\\[0.8ex]
{\rm (ii)}
For $c^{\nu}\in \lat^*\cap(Q^\infty)^{\veee}$ and $g_{\nu}\in\N$ we have
$\sum_{\nu}g_{\nu}\,\ues(c^{\nu}) \ge \ues(\sum_{\nu} g_{\nu}\, c^{\nu})$ 
in the sense of (\ref{tautco-2}).
\\[0.5ex]
{\rm (iii)}
$\sum_{i=1}^{\m}\eta_i(c)=\eta_0(c)$ and
$\,\sum_{\di=1}^{\m} \eta_i^*(c) = \ezs(c)$.
\end{lemma}

\begin{proof}
(iii) By definition of $\ku{\lambda^c}$ we have
$\sum_{i=1}^{\m}\lambda_i^cd^i=v(c)$. In particular:
$$
\begin{array}{rcl}
 \sum_{i=1}^{\m}\eta_i^*(c) &=&
 \sum_{i=1}^{\m}\eta_i(c)+\sum_{i=1}^{\m}(\ezs(c)-\eta_0(c))\cdot e_i[v(c)]
\\ 
&= & (-\sum_{i=1}^{\m}\langle \lambda_i^cd^i,c\rangle)+\ezs(c)-\eta_0(c)\\
&=&-\langle v(c),c\rangle+\ezs(c)-\eta_0(c)
=\eta_0(c)+\ezs(c)-\eta_0(c)=\ezs(c).
\end{array}
$$
The equality $\sum_{i=1}^{\m}\eta_i(c)=\eta_0(c)$ follows
from the previous argument by leaving out the $e[v(c)]$-terms.
\\[1ex]
(i) Now, for $c\in\lat^*\cap(Q^\infty)^{\veee}$, we will show that
$\ues(c)\in V^*$ can be represented by an integral vector
of $\R^{\m}$ having only non-negative coordinates:
We choose some path along the edges of $Q$ passing $v^0=0,\ldots, v^p=v(c)$
and decreasing the value of $c$ at each step. This provides some vector
$\ku{\lambda^c}\in\Z^{\m}$ yielding $\ku{\eta}(c)$ with
$\eta_i(c)=-\lambda_i^c\langle d^i,c\rangle\geq 0$.
\\
Denote by $v^{j_0},\ldots,v^{j_q}$ ($\{j_0,\ldots,j_q\}\subseteq\{0,\ldots,
p\}$) the $\lat$-vertices on the path. 
Then, for $s=1,\ldots, q$, the edges
between $v^{j_{s-1}}$ and $v^{j_s}$ 
(say $d^{i_1},\ldots, d^{i_k}$) belong to the same ''component''. 
In particular, not the single $\es_{i_1}(c),\ldots, \es_{i_k}(c)$ but
only their sums have to be considered:
\begin{equation*}
 \sum_{\mu=1}^k\es_{i_{\mu}}(c)=\sum_{\mu=1}^k\eta_{i_{\mu}}(c)=\langle
 -\sum_{\mu=1}^k\lambda^c_{i_{\mu}} d^{i_{\mu}},c\rangle= \langle
 v^{j_{s-1}}-v^{j_s},c\rangle\in\N.
\end{equation*}
If $v(c)$ belongs to the lattice $\lat$, then we are done. Otherwise,
there might be at most one non-integer coordinate (assigned to
$v(c)\notin\lat$)
in $\ues(c)$. However, this cannot be the case, since
the sum taken over all coordinates 
of $\ues(c)$ yields the integer $\ezs(c)$.
\\[1ex]
(ii) 
We define the following paths through the 1-skeleton of $Q$:
\vspace{-2ex}
\begin{itemize}
\item
$\underline{\lambda}:=$ path from $0\in Q$ to $v(\sum_{\nu} g_{\nu}\, c^{\nu})\in Q$,
\vspace{0.5ex}
\item
$\underline{\mu}^{\nu}:=$ path from $v(\sum_{\nu} g_{\nu}\, c^{\nu})\in Q$ to $v(c^{\nu})\in Q$
such that $\mu_\di^{\nu} \langle d^\di , c^{\nu} \rangle \le 0$ for each
$i=1,\dots,\m$.
\vspace{-2ex}
\end{itemize}
Then $\underline{\lambda}^{\nu} := \underline{\lambda} + \underline{\mu}^{\nu}$ is a
path from $0\in Q$ to $v(c^{\nu})$, and for $\di=1,\dots,\m$ we obtain
\begin{eqnarray*}
\sum_{\nu} g_{\nu} \, \eta_\di  (c^{\nu}) - \eta_i  \left( \sum_{\nu} g_{\nu}\, c^{\nu} \right) &=&
-\sum_{\nu} g_{\nu}\, (\lambda_\di  + \mu^{\nu}_i )\, \langle d^i ,c^{\nu} \rangle
+ \lambda_i  \left\langle d^i ,\, \sum_{\nu} g_{\nu}\, c^{\nu} \right\rangle\\
&=& -\sum_{\nu} g_{\nu}\, \mu^{\nu}_i  \, \langle d^i , c^{\nu} \rangle \ge 0\, .
\end{eqnarray*}
This yields the (componentwise) inequality
\[
\sum_{\nu} g_{\nu}\,\ues(c^{\nu}) \,\geq\, \sum_{\nu} g_{\nu}\,\ku{\eta}(c^{\nu}) \,\ge\,
\ku{\eta}\, (\sum_{\nu} g_{\nu}\,c^{\nu})\,.
\]
On the other hand, $\ku{\eta}\, (\sum_{\nu} g_{\nu}\,c^{\nu})$ and
$\ues\, (\sum_{\nu} g_{\nu}\,c^{\nu})$ might differ in at most one coordinate
(assigned to $a(\sum_{\nu} g_{\nu}\,c^{\nu})$). If so, then by definition
of $\ues$ the latter one
equals the smallest integer not smaller than the first one. Hence, we are
done, since the left hand side of our inequality involves integers only.
\end{proof}

We obtain the following description of
$\tilde{C}(Q)^{\veee}$:

\begin{proposition} \label{tautco-prop}
{\rm (1)}
$\tilde{C}(Q)^{\veee} =
\left\{ \left. [c,\underline{\eta}]\in (Q^\infty)^{\veee} \times V^*
\subseteq \A^* \times V^* \, \right| \;
\underline{\eta} - \underline{\eta}(c) \in C(Q)^{\veee} \right\}$.
\\[1ex]
{\rm (2)}
In particular, $[c, \underline{\eta}(c)] \in \tilde{C}(Q)^{\veee}$;
it is the only preimage
of $[c, \eta_0(c)] \in \sigma^{\veee}$ 
via the surjection $i^{\veee}:
\tilde{C}(Q)^{\veee} \surj \sigma^{\veee}$.
Moreover, for $c\in \lat^*\cap (Q^\infty)^{\veee}$, it holds
$[c,\ues(c)]\in \tilde{C}(Q)^{\veee}\cap\tilde{M}$. These elements
are liftings of $[c,\ezs(c)]\in \sigma^{\veee}\cap M$ -- 
but, in general,
they are not the only ones.
\\[1ex]
{\rm (3)}
$[c^1,\ues(c^1)],\dots, [c^w,\ues(c^w)]$ and $C(Q)^{\veee}\cap
V^*_{\Z}$,
embedded as $[0,C(Q)^{\veee}]$, generate the semigroup
$\Gamma:=\{[c,\ku{\eta}]\in (\lat^*\cap (Q^\infty)^{\veee})\times V^*_{\Z}\,|\;
\ku{\eta}-\ues(c)\in C(Q)^{\veee}\}\,
\subseteq \,\tilde{C}(Q)^{\veee}\cap\tilde{M}$.
Moreover,
$\tilde{C}(Q)^{\veee}\cap\tilde{M}$ is the saturation of that subsemigroup.
\end{proposition}

\begin{proof}
The proof of (1) and (2) is similar to the proof of
\cite[Prop.\ (4.6)]{gorenstein}.
\\[1ex]
(3) First, the condition
$\ku{\eta}-\ues(c)\in C(Q)^{\veee}$ indeed describes a semigroup; this
is a consequence of (ii) of Lemma \ref{tautco-lem}.
On the other hand, let $[c,\ues(c)]$ be given. Using some representation
$[c,\ezs(c)]=\sum_{t=1}^w p_{\nu}\, [c^{\nu}, \ezs(c^{\nu})]$ $(p_{\nu}\in\N)$, we
obtain by the same lemma
\[
\sum_{\nu} p_{\nu}\,\ues(c^{\nu})-\ues(c) =
\sum_{\nu} p_{\nu}\,\ues(c^{\nu})-\ues(\sum_{\nu} p_{\nu}\,c^{\nu}) \in C(Q)^{\veee}\;
\mbox{ (or even} \ge 0).
\]
Since, at the same time, the sum taken over all coordinates of that difference
vanishes, the whole difference has to be zero.
Now we obtain
\begin{eqnarray*}
[c,\ku{\eta}] & = & [c,\ues(c)]+[0,\ku{\eta}-\ues(c)]\\
& = & \sum_{\nu}p_{\nu}[c^{\nu},\ues(c^{\nu})]+[0,\ku{\eta}-\ues(c)].
\end{eqnarray*}
Finally, for every $[c,\ku{\eta}]\in\taut^\veee$ 
with $\ku{\eta}-\ues(c)\notin C(Q)^{\veee}$
there exists a $k\in\N_{\ge 1}$ such that $\ues(k\cdot c)=\ku{\eta}(k\cdot c)$,
since $v(c)=a(k\cdot c)$ yields $\ku{\eta}(k\cdot c)=k\cdot\ku{\eta}(c)$
and $\ku{\eta}(c)\in\Q^{\m}$. Then we obtain
\[
k\cdot\ku{\eta}-\ues(k\cdot c)=k\cdot\ku{\eta}-\ku{\eta}(k\cdot c)\\
=k\cdot(\ku{\eta}-\ku{\eta}(c))\in C(Q)^{\veee}
\]
by part (i) of this proposition. 
\end{proof}

\subsection{}\label{tautco-counter}
Now we will provide an example for the case $\Gamma\not=\tautv\cap\tilde{M}$:
\begin{exaX}
Let $N=\Z^3$ be a lattice. Define the cone $\sigma$ by
\[
\sigma:=\langle 
(0,0,1),(6,-1,2),(5,0,1),
(5,1,1),(24,7,5),(6,5,2),(2,3,1)\rangle\subseteq\Q^3=N_{\Q}.
\]
We choose $R:=[0,0,1]\in M=\Z^3$ and obtain the following polygon $Q$:
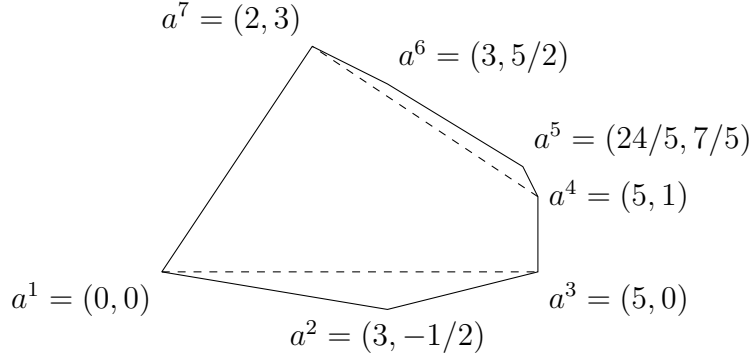
\begin{center}
\begin{tikzpicture}
\draw (0,0) node[anchor=north east] {$a^1=(0,0)$}
-- (3,-0.5) node[anchor=north] {$a^2=(3,-1/2)$}
-- (5,0) node[anchor=north west] {$a^3=(5,0)$}
-- (5,1) node[anchor=west] {$a^4=(5,1)$}
-- (4.8,1.4) node[anchor=south west] {$a^5=(24/5,7/5)$}
-- (3,2.5) node[anchor=south west] {$a^6=(3,5/2)$}
-- (2,3) node[anchor=south east] {$a^7=(2,3)$}
-- (0,0);
\draw[dashed] (5,1) -- (2,3);
\draw[dashed] (0,0) -- (5,0);
\end{tikzpicture}
\end{center}
The paths along the edges of $Q$ are denoted as follows:
\[
d^1=\twovec{3}{-\frac{1}{2}},\ d^2=\twovec{2}{\frac{1}{2}},
\ d^3=\twovec{0}{1},\ 
d^4=\twovec{-\frac{1}{5}}{\frac{2}{5}},
\]
\[ d^5=\twovec{-\frac{9}{5}}{\frac{11}{10}},\
d^6=\twovec{-1}{\frac{1}{2}},\ d^7=\twovec{-2}{-3}.
\]
Let us consider $V(Q)$. We identify $t_i$ and $t_j$ if the corresponding edges
have a common non-lattice vertex. Then $V(Q)$ as a subspace of $\R^4$ is the
kernel of the following matrix obtained by the 2-face equation of $Q$:
\[
\left(
\begin{array}{cccc} 5& 0 & -3 & -2\\ 0& 1 & 2 & -3\end{array}
\right) .
\]
It is generated by $\ult^1:=(13,0,15,10)$ and $\ult^2:=(2,15,0,5)$, and this
leads to $C(Q)=\Rpos\cdot\ult^1\oplus\Rpos\cdot\ult^2$.\\
Let $c:=[-1,-1]\in M$, then $v(c)=a^5$ and $\ku{\lambda}^c=(1,1,1,1,0,0,0)$.
Now we compute $\ku{\eta}(c)$ as described in (\ref{tautco-6}):
\[
\ku{\eta}(c)=[5/2,5/2,1,1/5,0,0,0].
\]
Since we only described $V(Q)$ as a subspace of $\R^4$, we can also denote
$\ku{\eta}(c)$ by $\ku{\eta}(c)=[5,1,1/5,0]$, which corresponds to taking
the sum on components of $Q$.
Let $\ku{\eta}:=[5,1,3/5,2/5]\in V^*$, so that it is also
contained in $V^*_{\Z}$: Since the first row of the matrix defining $V(Q)$
yields $5t_1=3t_3+2t_4$, the sum on the right hand side has $5$ as a divisor if
we only consider integral solutions. We could also regard
$\ku{\eta}$ as $[6,1,0,0]$ as element of $V_{\Z}^*$.
Let us consider $\ku{\eta}-\ku{\eta}(c)$:
\[
\ku{\eta}-\ku{\eta}(c)=[5-5,1-1,3/5-1/5,2/5-0]=[0,0,2/5,2/5].
\]
Obviously $\ku{\eta}-\ku{\eta}(c)$ is contained in $C(Q)^{\veee}$, as
it has positive entries only. Hence, 
$[c,\ku{\eta}]\in\tautv\cap\tilde{M}$.\\
We build up $\ues(c)$ as described in (\ref{tautco-6}): $\ues(c)=[5,1,1,0]$.
This yields
\[
\ku{\eta}-\ues(c)=[5-5,1-1,3/5-1,2/5-0]=[0,0,-2/5,2/5].
\]
Now we apply this to $\ult^1$:
\begin{eqnarray*}
\langle t_1,\ku{\eta}-\ues(c)\rangle & = 
& \langle (13,0,15,10) ,[0,0,-2/5,2/5]\rangle\\
& = & -2/5\cdot 15+2/5\cdot 10\\
& = & -6+4 = -2.
\end{eqnarray*}
Hence, we gain
$\ku{\eta}-\ues(c)\notin C(Q)^{\veee}$ and $[c,\ku{\eta}]\notin\Gamma$.
\end{exaX}
\begin{remark}
If one replaces the semigroup $\taut^{\veee}\cap\tilde{M}$ by its non-saturated
subgroup $\Gamma$ and $X$ by $X':=\Spec\C[\Gamma]$, respectively, 
then the diagram
of (\ref{tautco-taut}) becomes a fiber product diagram.
Moreover, 
$X'$ equals the scheme theoretical image of the map $X\to\C^w\times S$ 
induced
by the elements $[c^1,\ues(c^1)],\ldots,[c^w,\ues(c^w)]\in\Gamma$. And, in
return, $X$ is the normalization of $X'$.
\end{remark}

\section{A flat family over $\bar{\mathcal M}$}\label{flat}

We use the previous constructions to
provide a deformation of $Y$ over $\bar{\CM}$:

\begin{theorem} \label{flat-them}
Denote by $\bar{X}$ and $\bar{S}$ the scheme theoretical images of
$X$ and
$S$ in $\C^w\times\C^N$ and $\C^N$, respectively. Then
\\[1ex]
{\rm (1)} 
$X\rightarrow \bar{X}$ and $S\rightarrow\bar{S}$ are the
normalization maps,
\\[1ex]
{\rm (2)} 
$\pi:X'\rightarrow S$
induces a map $\bar{\pi}:\bar{X}\rightarrow \bar{S}$
such that $\pi$ can
be recovered from $\bar{\pi}$ via base change $S\rightarrow
\bar{S}$, and
\\[1ex]
{\rm (3)} 
restricting to $\CM\subseteq \bar{S}$
and composing with $\ell$ turns $\bar{\pi}$ into a family
\[
\bar{X} \times_{\bar{S}}\CM \stackrel{\bar{\pi}}{\longrightarrow}
\CM
\stackrel{\ell}{\longrightarrow}
\bar{\CM}\,.
\]
It is flat in $0\in \bar{\CM}\subseteq \C^{N-1}$, and the special
fiber equals
$Y$.
\vspace{-1ex}
\end{theorem}

The proof of this theorem will fill Section~\ref{flat}.

\subsection{}\label{flat-pf}
The ring of regular functions $A(\bar{S})$ is given as the image of
$\C[t_1,\dots,t_N] \rightarrow A(S)$. Since $\Z^N\surj V_{\Z}^*$ is
surjective, the rings
$A(\bar{S})\subseteq A(S) \subseteq \C[V^*_{\Z}]$ have the same field of
fractions.
\\[1ex]
On the other hand, while $t$-monomials with negative exponents might be
involved in $A(S)$, the surjectivity of $\R^N_{\geq 0} \surj C(Q)^{\veee}$
tells us that sufficiently high powers of those monomials always come from
$A(\bar{S})$. In particular, $A(S)$ is normal over $A(\bar{S})$.
\\[1ex]
$A(\bar{X})$ is given as the image
$A(\bar{X})=\im \, (\C[Z_1,\dots,Z_w,t_1,\dots,t_N]\rightarrow 
A(X'))$ with $Z_i\mapsto$ 
[monomial associated to $[c^i,\ezs(c^i)]$].
Since $A(X')$ is generated by these monomials over 
its subring $A(S)$,
cf.\ Proposition \ref{tautco-prop}(3)), 
the same arguments as for $S$ and $\bar{S}$ apply. 
Hence, Part (1) of the previous theorem is proved.

\subsection{}\label{flat-3} 
Denoting by $z_1,\dots,z_w,\,t$ the variables mapping
to the $A(Y)$-monomials with exponents 
$[c^1,\ezs(c^1)]$, $\dots$, $[c^w,\ezs(c^w)]$, $[0,1] \in
\sigma^{\veee} \cap M$,
respectively, we obtain the
following equations for $Y \subseteq \C^{w+1}$: 
\begin{eqnarray*}
f_{(a,b,\alpha,\beta)}(\ku{z},t) &:=& t^\alpha \,\prod_{t=1}^w z_{\nu}^{a_{\nu}} - 
t^\beta \, \prod_{t=1}^w z_{\nu}^{b_{\nu}} \\
&& \hspace{-1cm} \mbox{with }
\begin{array}[t]{l}
a,b\in \N^w: \; \sum_{\nu} a_{\nu}\, c^{\nu} = \sum_{\nu} b_{\nu}\, c^{\nu} \quad \mbox{ and}\\
\alpha, \beta \in \N: \; \sum_{\nu} a_{\nu}\,\ezs(c^{\nu}) + \alpha = 
\sum_{\nu} b_{\nu}\,\ezs(c^{\nu}) + \beta\, .
\vspace{2ex}
\end{array}
\end{eqnarray*}
Defining $c:=\sum_{\nu} a_{\nu}\, c^{\nu} = \sum_{\nu} b_{\nu}\, c^{\nu}$, we 
can lift them to the following elements of
$A(\bar{S})[Z_1,\dots,Z_w]$
(described by using liftings to
$\,\C[Z_1,\dots,Z_w,\,t_1,\dots,t_N]$):
\[
F_{(a,b,\alpha,\beta)}(\ku{Z},\ku{t}) := f_{(a,b,\alpha,\beta)}(\ku{Z},t_1) -
\ku{Z}^{[c,\ues(c)]} \cdot
\left( \ku{t}^{\alpha e_1 +\sum_{\nu} a_{\nu} \ues(c^{\nu})} -
\ku{t}^{\beta e_1 +\sum_{\nu} b_{\nu} \ues(c^{\nu})} \right) \cdot
\ku{t}^{-\ues(c)}\, .
\vspace{2ex}
\]

\begin{remark}
(1) The symbol $\ku{Z}^{[c,\ues(c)]}$ means $\prod_{v=1}^w Z_{\nu}^{p_{\nu}}$
with natural
numbers
$p_{\nu}\in \N$ such that
$[c,\ues(c)] = \sum_{\nu} p_{\nu}\, [c^{\nu}, \ues(c^{\nu})]$, or equivalently
$[c,\ezs(c)] = \sum_{\nu} p_{\nu}\, [c^{\nu}, \ezs(c^{\nu})]$.
This condition does not determine the coefficients $p_{\nu}$ uniquely
-- choose one of the possibilities.
Choosing other coefficients $q_{\nu}$ with the same property yields
\[
Z_1^{p_1}\cdot\dots\cdot Z_w^{p_w} - Z_1^{q_1}\cdot\dots\cdot
Z_w^{q_w} =
F_{(p,q,0,0)}(\ku{Z},\ku{t}) = f_{(p,q,0,0)}(\ku{Z},t)\, .
\]
(2) By part (iii) of Lemma \ref{tautco-lem}, we have $\sum_{\nu}
a_{\nu}\ues(c^{\nu}),\,
\sum_{\nu} b_{\nu}\ues(c^{\nu}) \ge \ues(c)$ in the sense of (\ref{tautco-2}).
In particular, representatives of the $\ues$'s can be chosen such
that
all $t$-exponents occurring in monomials of $F$ are non-negative integers,
i.e.\
$F$ indeed defines an element of $A(\bar{S})[Z_1,\dots,Z_w]$.
\end{remark}

\begin{lemma}
The polynomials $F_{(a,b,\alpha,\beta)}$ generate
$\mbox{Ker}\,(A(\bar{S})[\ku{Z}]
\rightarrow A(X'))$, i.e.\ they can be used as equations for 
$\bar{X} \subseteq \C^w 
\times \bar{S}$.
\vspace{-2ex}
\end{lemma}

\begin{proof}
Mapping $F$ into $A(X')= \oplus_{[c,\eta]} \,\C \,x^{[c,\eta]}\;$ 
($[c,\eta]$ runs through
all elements of $\Gamma \cap (\lat^*\times V_{\Z}^*)$;
$Z_{\nu} \mapsto x^{[c^{\nu}\!, \ues(c^{\nu})]},\, t_\di  \mapsto x^{[0,e_\di]}$)
yields
\begin{eqnarray*}
F_{(a,b,\alpha,\beta)} &=&
\begin{array}[t]{r}
\left( t_1^\alpha \, \prod_{\nu} Z_{\nu}^{a_{\nu}} -
\ku{Z}^{[c,\ues(c)]} \, \ku{t}^{\alpha e_1 + \sum_{\nu} a_{\nu}\ues(c^{\nu}) -
\ues(c)} \right) - \qquad\qquad\qquad\\
- 
\left( 
t_1^\beta \, \prod_{\nu} Z_{\nu}^{b_{\nu}} -
\ku{Z}^{[c,\ues(c)]} \, \ku{t}^{\beta e_1 + \sum_{\nu} b_{\nu}\ues(c^{\nu}) - \ues(c)}
\right)
\end{array}\\
& \mapsto &
\begin{array}[t]{r}
\left( x^{\alpha [0,e_1] + \sum_{\nu} a_{\nu} [c^{\nu}\!,\ues(c^{\nu})]} -
x^{[c,\ues(c)] + \alpha [0,e_1] + \sum_{\nu} a_{\nu} [0,\ues(c^{\nu})] - [0,\ues(c)]}
\right) - \quad\\
- \left( x^{\beta [0,e_1] + \sum_{\nu} b_{\nu} [c^{\nu}\!,\ues(c^{\nu})]} -
x^{[c,\ues(c)] + \beta [0,e_1] + \sum_{\nu} b_{\nu} [0,\ues(c^{\nu})] - [0,\ues(c)]}
\right)
\end{array}\\
& = & \; 0 \; - \; 0\, = \,0 \,.
\end{eqnarray*}
On the other hand, $\mbox{Ker}\,(A(\bar{S})[\ku{Z}] \rightarrow A(X'))$ is
obviously generated by the binomials
\[
\begin{array}[t]{r}
\ku{t}^{\ku{\eta}}\, Z_1^{a_1}\cdot\dots\cdot Z_w^{a_w} -
\ku{t}^{\ku{\mu}}\, Z_1^{b_1}\cdot\dots\cdot Z_w^{b_w} \quad 
\mbox{ such that}\hspace{4cm}
\vspace{1ex}\\
\begin{array}[t]{l}
\sum_{\nu} a_{\nu} [c^{\nu},\ues(c^{\nu})] + [0,\ku{\eta}] =
\sum_{\nu} b_{\nu} [c^{\nu},\ues(c^{\nu})] + [0,\ku{\mu}]  \, , \vspace{1ex}\\
\mbox{i.e.\ }
\begin{array}[t]{l}
c:= \sum_{\nu} a_{\nu} \, c^{\nu} = \sum_{\nu} b_{\nu}\, c^{\nu}\mbox{ and}\\
\sum_{\nu} a_{\nu}\, \ues(c^{\nu}) + \ku{\eta} = 
\sum_{\nu} b_{\nu}\, \ues(c^{\nu}) + \ku{\mu}\, .
\end{array}
\end{array}
\end{array}
\]
However,
\begin{eqnarray*}
\ku{t}^{\ku{\eta}}\,\ku{Z}^{a} -
\ku{t}^{\ku{\mu}}\,\ku{Z}^{b} &=&
\begin{array}[t]{r}
\ku{t}^{\ku{\eta}}\cdot
\left( \prod_{\nu} Z_{\nu}^{a_{\nu}} - \ku{Z}^{[c ,\ues(c)]} \, 
\ku{t}^{\sum_{\nu} a_{\nu} \ues(c^{\nu}) -
\ues(c)} \right) - \qquad\\
- \ku{t}^{\ku{\mu}}\cdot
\left( \prod_{\nu} Z_{\nu}^{b_{\nu}} - 
\ku{Z}^{[c ,\ues(c)]} \, \ku{t}^{\sum_{\nu} b_{\nu} \ues(c^{\nu}) -
\ues(c)} \right) 
\end{array}\\
& = &
\ku{t}^{\ku{\eta}}\cdot F_{(a,p,0,\alpha)} -
\ku{t}^{\ku{\mu}}\cdot F_{(b,p,0,\beta)}
\end{eqnarray*}
with $p\in \N^w$ such that $\sum_{\nu} p_{\nu} [c^{\nu}, \ues(c^{\nu})] = [c, \ues(c)]$,
$\alpha = \sum_{\nu} a_{\nu} \ezs(c^{\nu}) - \ezs(c)$, and
$\beta = \sum_{\nu} b_{\nu} \ezs(c^{\nu}) - \ezs(c)$.
\end{proof}

Using exponents $\ku{\eta}, \ku{\mu} \in \Z^N$ (instead of $\N^N$), the
binomials
$\ku{t}^{\ku{\eta}}\,\ku{Z}^{a} -
\ku{t}^{\ku{\mu}}\,\ku{Z}^{b}$ generate the kernel of the map
\[
A(S)[\ku{Z}] = A(\bar{S})[\ku{Z}] \otimes_{A(\bar{S})} A(S) \surj
A(\bar{X})  \otimes_{A(\bar{S})} A(S) \surj A(X')\, .
\]
Since $\ku{Z}^a \otimes \ku{t}^{\ku{\eta}} -
\ku{Z}^b \otimes \ku{t}^{\ku{\mu}} =
\ku{Z}^{[c,\ues(c)]} \otimes \left(
\ku{t}^{\sum_{\nu} a_{\nu} \ues(c^{\nu}) - \ues(c) + \ku{\eta}} -
\ku{t}^{\sum_{\nu} b_{\nu} \ues(c^{\nu}) - \ues(c) + \ku{\mu}} \right) = 0$
in $A(\bar{X})  \otimes_{A(\bar{S})} A(S)$, this implies that the
surjection $A(\bar{X})  \otimes_{A(\bar{S})} A(S) \surj A(X')$ is
injective, too. In particular, part (2) of our theorem is proved.

We are going to use the following well known criterion of flatness:

\begin{theorem} \label{flat-crit}
Let $\tilde{\pi}:\tilde{X}\hookrightarrow 
\C^{w+1}\times\bar{\CM}\surj \bar{\CM}$ 
be a map with special fiber $Y= \tilde{\pi}^{-1}(0)$; 
in particular, $Y\subseteq \C^{w+1}$
is defined by the restrictions to $0\in\bar{\CM}$ of the equations defining
$\tilde{X}\subseteq \C^{w+1}\times\bar{\CM}$. Then $\tilde{\pi}$ is flat, if
and only if each linear relation between the (restricted) equations for $Y$
lifts to some linear relation between the original equations for $\tilde{X}$.
\vspace{-2ex}
\end{theorem}

\begin{proof}
According to
\cite[(20.C), Theorem 49]{Ma}, 
flatness of $\tilde{\pi}$ in $0\in\bar{\CM}$ is
equivalent to the vanishing of
$\mbox{Tor}_1^{\CO_{\bar{\CM},0}}((\tilde{\pi}_{*}\CO_{\tilde{X}})_0,
\C)$ where $\C$ becomes an $\CO_{\bar{\CM},0}$-module via evaluating in
$0\in\bar{\CM}$.
\\
Using the embedding $\tilde{X}\hookrightarrow \C^{w+1}\times\bar{\CM}$
(together with the defining equations and linear relations between them) 
we obtain an
$\CO_{\bar{\CM},0}[Z_0,\dots,Z_w]$-free 
(hence $\CO_{\bar{\CM},0}$-free) resolution
of $(\tilde{\pi}_{*}\CO_{\tilde{X}})_0$ up to the second term. Now,
the condition that relations between $Y$-equations lift to those between
$\tilde{X}$-equations is equivalent to the fact that our (partial) resolution
remains exact under $\otimes_{\CO_{\bar{\CM},0}}\,\C$.
\vspace{-2ex}
\end{proof}

For our special situation take $\tilde{X}:= \bar{X} \times_{\bar{S}}\CM$ (and
$\bar{\CM}:= \bar{\CM},\, Y:= Y$); in (\ref{flat-3}) 
we have seen how
the equations defining $Y\hookrightarrow \C^w\times\C$ can be lifted to those
defining $\bar{X}\hookrightarrow \C^w\times\bar{S}$, hence
$\bar{X} \times_{\bar{S}}\CM \hookrightarrow \C^w\times \CM
\stackrel{\sim}{\rightarrow} \C^w\times\C\times\bar{\CM}$.\\
In particular, to show (3) of Theorem \ref{flat-them}, we
only have to take the linear relations
between the $f_{(a,b,\alpha,\beta)}$'s and lift them to relations between
the $F_{(a,b,\alpha,\beta)}$'s.

\subsection{}\label{flat-three-Rel} 
According to the special shape of our generator set $E$,
there are three types of relations between the 
$f_{(a,b,\alpha,\beta)}$'s:
\vspace{-1.7ex}
\begin{itemize}
\item[(i)]
$f_{(a,r,\alpha,\gamma)} + f_{(r,b,\gamma,\beta)} = f_{(a,b,\alpha,\beta)}$\\
with
$\begin{array}[t]{l}
\sum_{\nu}a_{\nu}c^{\nu} = \sum_{\nu} r_{\nu} c^{\nu} = \sum b_{\nu} c^{\nu}\; \mbox{ and}\\
\sum_{\nu} a_{\nu}\ezs(c^{\nu}) +\alpha = \sum_{\nu}r_{\nu}\ezs(c^{\nu}) +\gamma =
\sum_{\nu}b_{\nu}\ezs(c^{\nu}) + \beta\, .
\end{array}$\\
For this relation, the same equation between the $F$'s is true.
\vspace{0.8ex}
\item[(ii)]
$t\cdot f_{(a,b,\alpha,\beta)} = f_{(a,b,\alpha+1,\beta+1)}\;$ lifts to
$\;t_1\cdot F_{(a,b,\alpha,\beta)} = F_{(a,b,\alpha+1,\beta+1)}$.
\vspace{0.8ex}
\item[(iii)]
$\ku{z}^r\cdot f_{(a,b,\alpha,\beta)} = f_{(a+r,b+r,\alpha,\beta)}$.
\end{itemize}

With $c:= \sum_{\nu}a_{\nu}c^{\nu} = \sum_{\nu}b_{\nu}c^{\nu},\; \tilde{c}:= 
c+ \sum_{\nu}r_{\nu}c^{\nu}$ we obtain for (iii)
\[
\begin{array}{l}
\ku{Z}^r\cdot F_{(a,b,\alpha,\beta)} - F_{(a+r,b+r,\alpha,\beta)} =
\vspace{1ex}\\
\qquad\begin{array}[t]{r}
=\ku{Z}^{[\tilde{c},\ues(\tilde{c})]}
\cdot
\left( \ku{t}^{ \alpha e_1 + \sum_{\nu}a_{\nu}\ues(c^{\nu}) +
\sum_{\nu}r_{\nu}\ues(c^{\nu})} - \ku{t}^{ \beta e_1 + \sum_{\nu}b_{\nu}\ues(c^{\nu}) +
\sum_{\nu}r_{\nu}\ues(c^{\nu})} \right)
\cdot \ku{t}^{-\ues(\tilde{c})} - \\
- \ku{Z}^{[c,\ues(c)]}\, \ku{Z}^r\cdot
\left( \ku{t}^{ \alpha e_1 + \sum_{\nu}a_{\nu}\ues(c^{\nu})} -
\ku{t}^{ \beta e_1 + \sum_{\nu}b_{\nu}\ues(c^{\nu})} \right)\cdot
\ku{t}^{-\ues(c)}
\end{array}
\vspace{1ex}\\
\qquad \begin{array}[t]{r} =
\left( \ku{t}^{ \alpha e_1 + \sum_{\nu}a_{\nu}\ues(c^{\nu})-\ues(c)} -
\ku{t}^{ \beta e_1 + \sum_{\nu}b_{\nu}\ues(c^{\nu})-\ues(c)} \right)\cdot
\hspace{4cm}\\
\left( \ku{t}^{\ues(c)+\sum_{\nu}r_{\nu}\ues(c^{\nu}) - \ues(\tilde{c})}
\ku{Z}^{[\tilde{c},\ues(\tilde{c})]} - \ku{Z}^{[c,\ues(c)]}
\ku{Z}^r \right)\,.
\end{array}
\end{array}
\]
Now, the inequalities
\[
\sum_{\nu}a_{\nu}\ues(c^{\nu}),\,\sum_{\nu}b_{\nu}\ues(c^{\nu}) \ge \ues(c)
\;\mbox{ and }\;
\ues(c)+\sum_{\nu}r_{\nu}\ues(c^{\nu}) - \ues(\tilde{c}) \ge 0
\]
imply that
the first factor is contained in the ideal defining
$0\in \bar{\CM}$ and that
the second factor is an equation of $\bar{X}\subseteq
\C^w\times\bar{S}$ (called $F_{(q,p+r,\xi,0)}$ in
(\ref{obst-threeRel})).
In particular, we have found a lift for the third relation, too.
The proof of Theorem \ref{flat-them} is complete.

\section{The Kodaira-Spencer map}\label{kodaira}

\subsection{}\label{kodSpencE}
To each vertex $v^j\in Q$ we associate the subset
\[
 E_j:=E_{v^j}:=\{ [c,\ezs(c)]\in E|\ 
\langle v^j,c\rangle +\ezs(c)<1\}.
\]
Additionally define the sets
\[
E_0:=\bigcup_j E_j,\quad E_{ij}:=E_i\cap E_j.
\]
Let $r=[c,\ezs(c)]\in E$ be given. Then we have
\[
\langle v^j,c\rangle +\ezs(c)=
\langle (v^j,1)\ ,\ [c,\ezs(c)]\rangle
= \frac{\langle a^j,r\rangle}{\langle a^j,R\rangle}
\]
and we obtain the following alternative description of $E_j$:
\[
E_j:=\{r\in E|\ \langle a^j,r\rangle\ <\
\langle a^j, R\rangle\}.
\]
In other words: The primitive generators $a^j$ of $\sigma$ define the
facets of the dual cone $\sigv$, i.e. they define hyperplanes such that
$\sigv$ is the intersection of the halfspaces above these hyperplanes:
\[
\sigv=\bigcup_j\{m\in M_{\R}|\
\langle a^j, m\rangle \ge 0\}\subseteq M_{\R}.
\]
Now $E^j$ contains those elements of $E$ that are closer to the facet of
$\sigma$
defined by $a^j$ than $R$.\\
We also get the following alternative
description of $\ues(c)$ compared with its definition in 
(\ref{tautco-6}):

\begin{lemma}
 Assume that $[c,\ezs(c)]$ is contained in $E_j$. Then 
\[
\ues(c)=\langle v^j_\bullet,-c\rangle +
(\langle v^j,\,c\rangle+\ezs(c) )\cdot e[v^j]
\]
where
$v^j_\bullet$ denotes the map assigning $\ult\in V(Q)$ the vertex
$v^j_{\ult}$ of the (generalized) Minkowski summand $Q_{\ult}$.
\end{lemma}

\begin{proof}
If $v^j\in\lat$, then the condition 
$\langle v^j,c\rangle +\ezs(c)<1$
is equivalent to $\langle v^j,-c\rangle = \eta_0(c)=\ezs(c)$. Hence, the
second summand in our formula vanishes, and we are done.\\
On the other hand, if $v^j\notin \lat$, then there is not any lattice point
contained in the strip $\langle v^j, c\rangle \ge \langle \bullet,c\rangle
> \langle v(c),c\rangle$. In particular,
every edge on the path from
$v^j$ to $v(c)$ (decreasing the $c$-value at each step) belongs
to the ``component'' induced by $v^j$, cf.\ (\ref{def-eps}). 
Now, our
formula follows from the definition of $\ues(c)$.
\end{proof}

\subsection{}\label{kodSpencPair}
Denoting by $L(\kbb)$ the abelian group of $\Z$-linear relations
of the argument,
we consider the bilinear map
\[
\begin{array}{cccccl}
\Phi: & ^{\kd V_{\Z}}\!/_{\kd \Z\cdot\ku{1}}
& \times & L(\cup_j E_j)
& \longrightarrow & \Z\\
&\ult & , & q & \mapsto & 
\sum_{v,\di} t_\di \, q_{\nu}\, \es_i (c^{\nu})\,.
\end{array}
\]
It is correctly defined,
and we obtain $\Phi(\ult,q)=0$ for $q\in L(E_j)$. 
Indeed, 
\[
\begin{array}{rcl}
\Phi(\ult,q) &=& \sum_{\nu} q_{\nu}\,\langle \ult,\ues(c^{\nu})\rangle\\
&=&
\sum_{\nu} q_{\nu}\cdot \left(\langle v^j_{\ult},-c^{\nu}\rangle
+ (\langle v^j,c^{\nu}\rangle+ \ezs(c^{\nu}))\cdot t_{v^j} \right)\\
&=&
\langle v^j_{\ult}, -\sum_{\nu}q_{\nu}c^{\nu}\rangle +
\left( \sum_{\nu}q_{\nu}\,\ezs(c^{\nu})- \langle 
v^j,\sum_{\nu}q_{\nu}c^{\nu}\rangle \right)
\cdot t_{v^j}
\;=\;0\,.
\end{array}
\vspace{2ex}
\]

\begin{theorem}\label{ks-thm}
The Kodaira-Spencer map of the family $\bar{X}\times_{\bar{S}}{\CM} \to
\bar{\CM}$ of (\ref{flat-them}) equals the map
\vspace{-1ex}
\[
T_0\bar{\CM} \;=\; ^{\kd V_{\C}}\!\!\big/_{\kd \C\cdot\ku{1}}
\longrightarrow
\left( \left. ^{\kd L_{\C}(E\cap\partial\sigv)} \! \right/
\! \sum_j
L_{\C}(E_j) \right)^*
=T^1_Y(-R)
\]
induced by the previous pairing. 
Moreover, this map is an isomorphism.
\vspace{-1ex}
\end{theorem}

\begin{proof}
 Using the same symbol $\kI$ for the ideal $\kI\subseteq \C[t_1,\dots,t_N]$
as well as for the intersection $\kI\cap \C[t_i -t_j \,|\; 1\leq
i,j\leq N]$, cf.\ (\ref{def-higherdeg}),
our family corresponds to the flat 
$\C[t_i -t_j ]/{\kd \kI} $-module
$\C[\ku{Z},\ult]/{\kd (\kI, F_\bullet (\ku{Z}, \ult))}$.
Now, we fix a non-trivial tangent vector $\ult^0\in V_{\C}$. Via
$t_i \mapsto t+t_i ^0\,\keps$, it induces the infinitesimal family given by
the flat $\C[\keps]/\!_{\kd \keps^2}$-module
\[
A_{\ult^0} := \;\left.^{\kd \C[\ku{z}, t, \keps]}\! \right/ \!
_{\kd (\keps^2, F_\bullet(\ku{z}, t+\ult^0\,\keps))}\,.
\]
To obtain the associated $A(Y)$-linear map 
$I/{\kd I^2} \rightarrow A(Y)$
with $I:=(f_\bullet (\ku{z},t))$ denoting the ideal of $Y$ in 
$\C^{w+1}$,
we have to compute the images of $f_\bullet (\ku{z},t)$ in
$\keps \,A(Y) \subseteq
A_{\ult^0}$ and divide them by $\keps$:
Using the notation of (\ref{flat-3}), in
$A_{\ult^0}$ it holds true that
\begin{eqnarray*}
0 &=& F_{(a,b,\alpha,\beta)}(\ku{z}, t+\ult^0\,\keps)\\
&=& \!\begin{array}[t]{l}
f_{(a,b,\alpha,\beta)}(\ku{z}, t+t^0_1\,\keps) -\\[0.5ex]
\quad - \ku{z}^{[c,\ues(c)]}\cdot
\left( (t+\ult^0\,\keps)^{\alpha e_1 + \sum_{\nu} a_{\nu} \ues(c^{\nu})-\ues(c)}
- (t+\ult^0\,\keps)^{\beta e_1 + \sum_{\nu} b_{\nu} \ues(c^{\nu})-\ues(c)}
\right)\,.
\end{array}
\end{eqnarray*}
The relation $\keps^2=0$ yields
\[
f_{(a,b,\alpha,\beta)}(\ku{z}, t+t^0_1\keps) =
f_{(a,b,\alpha,\beta)}(\ku{z}, t) + \keps\cdot
(\alpha \,t^{\alpha-1} \,t_1^0 \,\ku{z}^a -
\beta \,t^{\beta-1}\, t_1^0 \,\ku{z}^b)\,,
\]
and similarly we can expand the other terms.
Eventually, we obtain
\begin{eqnarray*}
f_{(a,b,\alpha,\beta)}(\ku{z},t) &=& \!
\begin{array}[t]{l}
-\keps\,t_1^0\, (\alpha\, t^{\alpha -1}\,\ku{z}^a -
\beta\, t^{\beta -1}\,\ku{z}^b )\,+\,
\keps \, \ku{z}^{[c,\ues(c)]}\,
t^{\alpha+\sum_{\nu} a_{\nu} \ezs(c^{\nu})-\ezs(c)-1}\cdot
\vspace{1ex}\\
\quad\qquad\qquad\qquad\qquad\cdot
\left[
t_1^0\, (\alpha-\beta)
+\sum_i t_i ^0\, \left(\sum_{\nu} (a_{\nu}-b_{\nu})\es_i (c^{\nu}) \right) \right]
\end{array}
\vspace{1ex}\\
&=&
\, \keps\cdot x^{\sum_{\nu} a_{\nu} [c^{\nu},\ezs(c^{\nu})] +[0,\alpha-1]} \cdot
\left( \sum_i t_i ^0\, \left(\sum_{\nu} (a_{\nu}-b_{\nu})\es_i (c^{\nu}) \right) \right)\,.
\end{eqnarray*}
Note that, in $\keps \, A(Y)$, 
we were able to replace the variables $t$ and $z_{\nu}$
by $x^{[\ku{0},1]}$ and $x^{[c^{\nu},\ezs(c^{\nu})]}$, respectively.
\\[1ex]
On the other hand, the explicit description of
$T^1_Y(-R)$ as $L_{\C}(E\cap\partial\sigv)/\sum_jL_{\C}(E_j)$
was given in \cite[Theorem (3.4)]{al3}. It even says
that the map $I/I^2\to A(Y)$ with
$\big(t^\alpha\,\ku{z}^a - t^\beta\,\ku{z}^b\big) \mapsto 
\left( \sum_{i,v} t_i ^0\,  (a_{\nu}-b_{\nu})\es_i (c^{\nu}) \right)
\cdot x^{\sum_{\nu} a_{\nu} [c^{\nu},\ezs(c^{\nu})] +[0,\alpha-1]}$
corresponds to $q\mapsto \sum_{i,v} t_i ^0\,q_{\nu}\,\es_i (c^{\nu})
=\Phi(\ku{t}^0,q)$.

In \cite{flip} it was already proven that there is an isomorphism 
$\Psi:T_Y^1(-R)\to\VQ/\ku{1}$ if $Y$ is smooth in codimension two. Now we want
to show that the composition $\Psi\circ\Phi$ yields the identity on
$\VQ/\ku{1}$. Thus we have to take a closer look at the construction of
$\Psi$.\par
Let us switch from the notion of $L(E_j)$ to the notion of $\spa(E_j)$.
The advantage lies in that $\spa(E_j)$ is much easier to
describe than $L(E_j)$:
\begin{remark}
\[
\spa_{\R}E_j =
\begin{cases}
0 & \langle a^j, R\rangle =0\\
(a^j)^{\bot} & \langle a^j, R\rangle =1\\
M_{\R} & \langle a^j, R\rangle \ge 2
\end{cases}.
\]
\end{remark}
To change between the two notions, let $q\in L(E_0)$ be given. Then
decompose $q=\sum_j q^j$ with $q^j\in\Z^{E_j}$. Define
$w^j:=\sum_{\nu}q^j_{\nu}r^{\nu},\ r^\nu\in E$
so that the vector $(w^1,\ldots, w^M)$
is contained
in $\ker(\oplus_j\spa E_j\to M)$. \par
Let $\tau < \sigma$ be a face. We define the following set
\[
E_{\tau}:=\bigcap_{a^j\in\tau}E_j
\]
and obtain a complex $\spa(E)_{\bullet}$ with
\[
\spa_{-k}(E)=\bigoplus_{
\begin{array}{c}
\tau\mbox{ a face of }\sigma\\
\dim{\tau}=k
\end{array}} E_{\tau}
\]
and the obvious differentials. Now the dual complex yields
the following description of $T_Y^1(-R)$:
\begin{theorem} {\bf (\cite{al3} (6.1))} 
The homogenous piece of $T_Y^1$ in degree $-R$ is given by
\[
T_Y^1(-R)\ =\ H^1\left(\spa(E)_{\bullet}^*\otimes_{\Z}\C\right),
\]
i.e. $T_Y^1(-R)$ equals the complexification of the cohomology of
the subsequence
\[
N_{\R}\to\bigoplus_j(\spa_{\R}E_j)^*
\to \bigoplus_{\langle a^j,a^j\rangle<\sigma}(\spa_{\R}E_{ij})^*
\] 
of the dual complex to $\spa(E)_{\bullet}$.
\end{theorem}
Given an element $b\in T_Y^1(-R)$, we can
build an element $\ult\in\VQ/\ku{1}$. First we will show how to build
$\ult\in\VQ$ from a given $b\in\bigoplus_j(\spa_{\R}E_j)^*$. Then we will show
that the action of $N_{\R}$ equals the action of $\R\cdot\ku{1}$ on $\VQ$.
\begin{asparaitem}
\item[{\it Step 1:}] By the above remark, we can represent 
$b\in\bigoplus_j(\spa_{\R}E_j)^*$ by a family of
\begin{itemize}
\item $b^j\in N_{\R}$ if $\langle a^j,R\rangle\ge 2$ and
\item $b^j\in N_{\R}/\R\cdot a^j$ if $\langle a^j,R\rangle=1$.
\end{itemize}
We will only consider the $b^j$ for $\langle a^j, R\rangle \ge 1$, otherwise
$\spa E_j$ will be zero. This corresponds to the fact that $v^j=a^j/\langle
a^j,R\rangle$ is
not a vertex of $Q$.\par
Dividing by the image of $N_{\R}$ means shifting the family by a common
vector $c\in N_{\R}$. The condition of our family $\{b^j\}$ mapping
onto $0$ means that $b^j$ and $b^j$ have to be equal on $\spa_{\R}E_{ij}$ for
each compact edge $\ko{v^j,v^j}<Q$. Since
\[
(a^j, a^k)^{\bot}\subseteq\spa_{\R}E_{jk}\subseteq\spa_{\R}E_j\cap\spa_{\R}E_k
\]
we obtain $b^j-b^k\in\R a^j+\R a^k$.
\item[{\it Step 2:}] Let us introduce new coordinates
\[
\ko{b}^j:= b^j-\langle b^j,R\rangle v^j\ \in\ R^{\bot}.
\]
The condition $b^j-b^k\in\R a^j+\R a^j$ changes into the condition
$\ko{b}^j-\ko{b}^k\in\R v^j+\R v^k$. We assume
$\langle a^j,R\rangle,\langle a^k,R\rangle\not= 0$, i.e.
$a^j,a^k\notin R^{\bot}$. On the other hand, we know
$\ko{b}^j,\ko{b}^k\in R^{\bot}$, hence $\ko{b}^j-\ko{b}^k\in R^{\bot}$.
This yields
\[
\ko{b}^j-\ko{b}^k\in(\R v^j+\R v^k)\cap R^{\bot}
=\R (v^j- v^k).
\]
Thus we obtain
\[
\ko{b}^j-\ko{b}^k = t_{jk}\cdot(v^j- v^k).
\]
Now collect these $t_{ij}$ for each compact edge 
$\ko{v^j,v^k}<Q$. Together they yield an element $\ult_b\in\R^N$.
\item[{\it Step 3:}]
Consider shifting the family by a common vector $c\in N_{\R}$, i.e. 
$b^{j'}:=b^j+c$. We obtain
\begin{eqnarray*}
t_{jk}'(v^j- v^k)
 & = & \ko{b}^{j'}-\ko{b}^{k'}\\
& = & (b^j+c-\langle b^j, R\rangle v^j-\langle c, R\rangle v^j)
-(b^k+c-\langle b^k, R\rangle v^k-\langle c, R\rangle v^k)\\
& = & \ko{b}^j-\ko{b}^k-\langle c, R\rangle\cdot(v^j- v^k)
= (t_{jk}-\langle c, R\rangle)\cdot(v^j- v^k).
\end{eqnarray*}
Hence, the action of $c\in N_{\R}$ comes down to an action of
$\langle c, R\rangle$ only, and we obtain $\ult_b\in\R^N/\ku{1}$.
\item[{\it Step 4:}]
It is rather easy to see that $\ult_b$ satisfies the 2-face equations of $\VQ$.
In \cite{flip} (2.7) it is proven that $\ult_b$ also satisfies the equations
given by non lattice vertices of $Q$ since $Y$
is smooth in codimension two. We obtain the following Corollary:
\begin{corX}
{\bf (\cite{flip} (2.6))}If $Y$ is smooth in codimension two 
\[
\begin{array}{cccc}
\Psi: & T_Y^1(-R) & \longrightarrow & \VQ_{\C}/\ku{1}\\
& b & \mapsto & \ult_b
\end{array}
\]
is an isomorphism.
\end{corX}
\item[{\it Step 5:}]
Let us now combine $\Phi$ with this isomorphism.
Denoting by $t_j$ the coordinate of $\ku{t}$ corresponding to the
component arising from a non-lattice vertex $v^j$ and
defining $g^j:=-\sum_{\nu}q^j_{\nu}[c^{\nu},\ezs(c^{\nu})]$,
we obtain
\begin{eqnarray*}
\Phi(\ult,q)=\sum_j\Phi(\ult,q^j)&
= &\sum_{j,\nu}q_{\nu}^j\langle \ult,\ues(c^{\nu})\rangle\\
& = &\sum_{j,\nu}q_{\nu}^j\langle v_{\ult}^j,-c^{\nu}\rangle
+\sum_{j,\nu}q_{\nu}^j[\ezs(c^{\nu})+\langle v^j,-c^{\nu}\rangle]\cdot
t_j\\
& = & \sum_{j,\nu}\langle(v_{\ult}^j,0),g^j\rangle
-\langle(v^j,1),g^j\rangle\cdot t_j
\end{eqnarray*}
i.e.\ $\Phi$ assigns to $\ult$ exactly the vertices of the corresponding
Minkowski summand $Q_{\ult}$. Thus, applying $\Psi$ to
$\Phi(\ult)$
yields the identity.
\end{asparaitem}
\end{proof}

\section{The obstruction map}\label{obst}

Now we can approach the main goal of this paper:
\begin{theorem}\label{obst-mainthm}
The family of Theorem \ref{flat-them}
with base space $\bar{\mathcal M}$
is the versal deformation of $Y$ of degree $-R$.
\end{theorem}
By \cite{versell} we know that a deformation is versal if the Kodaira-Spencer-map
is an isomorphism and the Obstruction map is injective. In section \ref{kodaira}
we proved the first condition for the degree-$R$-part of $T_Y^1$. The following
section will prove the second condition, i.e. the injectivity of the
obstruction map.

\subsection{}\label{obst-1} Dealing with obstructions in the 
deformation theory of $Y$
involves the $A(Y)$-module $T_Y^2$. Usually it is defined in the following way:
Let
\[
m:=\{([a,\alpha ],[b,\beta ])\in\N^{w+1}\times \N^{w+1}|
\begin{array}[t]{rcl}
\sum_\nu a_{\nu}c^{\nu}& =& \sum_\nu b_{\nu}c^{\nu}\hspace{0.7em}
\mbox{ and}\\
\sum_\nu a_{\nu}\ezs (c^{\nu})+\alpha & =& 
\sum_\nu b_{\nu}\ezs (c^{\nu})+\beta \}
\end{array}
\]
denote the set parametrizing the equations $f_{(a,b,\alpha,\beta)}$ generating
the ideal $I\subseteq\C[\ku{z},t]$.
Then
\[
\CR:=\ker{(\varphi:\C[\ku{z},t]^m\twoheadrightarrow I)}
\]
is the module of linear relations between these equations; it contains the
submodule $\RI_0$ of
the so-called Koszul relations, i.e.\ those of the form 
$f_j\cdot e_i-f_i\cdot e_j$ 
where $f_i$, $f_j$ are generators of 
$I$ and $e_i,\ e_j$ are their corresponding preimages under
$\varphi$. 

\begin{defX}
$\;T_Y^2:=\;^{\displaystyle \Hom(^{\displaystyle \RI}\! \left/ \!
_{\displaystyle \RI_0}\right., A)} \! \left/ \!
_{\displaystyle \Hom(\C[\ku{z},t]^m,A)} \right.\; .
\vspace{-1ex}
$
\end{defX}

Recall that $R=[\ku{0},1]$.
To obtain information about $T^2$ not only in degree $-R$ but also
in its multiples $k\cdot R,\; k\ge 2$, we define,
analogously to $E_j$, the following sets:
\[
E_j^k:=
\{[c^{\nu},\ezs(c^{\nu})]\,|\; \langle a^j ,c^{\nu}\rangle + \ezs(c^{\nu}) 
< k\}
\cup \{R\} \subseteq \sigv\cap M\,.
\vspace{1ex}
\]
For the following theorem it is very important that $\sigma$ has smooth
two-dimensional faces, i.e. that $Y$ is smooth in codimension two:
\begin{theorem}
\cite{al3} The vector space $T_Y^2$ is $M$-graded, and
in degree $-kR$ it equals
\[ T_Y^2(-kR)=
 \left(
\frac{
\ker(\oplus_jL_{\C}(E^k_j)\to L_{\C}(E))}
{\im(\oplus_{\langle v^i,v^j\rangle <Q} L_{\C}(E^k_i\cap E^k_j)\to
\oplus_i L_{\C}(E^k_i))}
\right)^*.
\]
\end{theorem}

\subsection{}\label{obst-2} 
In this section we build up the so-called obstruction map.
It detects all infinitesimal extensions of our family over $\bar{\CM}$ to
a flat family over some larger base space. By $\kI$ let us denote
\[
\kI:=(g_{\ku{d},k}(\ult-t_1\, |\, \ku{d}\in\VQ^{\bot}\cap\Z^N,\,k\ge 1)
\subseteq\C[t_i-t_j]
\]
the homogenous ideal of the base space $\bar{\CM}$.
Let $\kI_1$ denote the degree $1$ part of $\kI$.
We define the subideal $\ktI\subseteq \kI$ by:
\[ \ktI= (t_i-t_j)_{i,j}\cdot\kI + \kI_1\cdot\C [t_i-t_j]\subseteq
\C[t_i-t_j|\ 1\le i,j\le\m].\]
Then $W:=\kI/\ktI$ is a
finite-dimensional, $\Z$-graded vector space. It comes as the kernel in the
exact sequence
\[
0\to W \to\,
 ^{\displaystyle \C[t_i-t_j]} \left/ _{\displaystyle \ktI}\right. 
\to\,
^{\displaystyle \C[t_i-t_j]} \left/ _{\displaystyle \kI}\right. \to 0.
\]
Identifying $t$ with $t_1$ and $\ku{z}$ with $\ku{Z}$, the tensor product with
$\C[\ku{z},t]$ over $\C$ yields the important exact sequence
\[
0\to W\otimes\C[\ku{z},t] \to\,
 ^{\displaystyle \C[\ku{Z},\ult]} \left/ 
 _{\displaystyle \ktI\cdot\C[\ku{Z},\ult]}\right. \to\,
^{\displaystyle \C[\ku{Z},\ult]} \left/ 
_{\displaystyle \kI\cdot \C[\ku{Z},\ult]}\right. \to 0.
\]
Now, let $s$ be any relation with coefficients in $\C[\ku{z},t]$ between the
equations $f_{(a,b,\alpha,\beta)}$, i.e.
\[
\sum s_{(a,b,\alpha,\beta)}f_{(a,b,\alpha,\beta)}=0\ \mbox{ in }\C[\ku{z},t]. 
\]
By flatness of our family, cf.\ (\ref{flat-three-Rel}), the components of $s$ can be 
lifted to $\C[\ku{Z},\ult]$ obtaining an $\tilde{s}$, such that
\[\sum \tilde{s}_{(a,b,\alpha,\beta)}F_{(a,b,\alpha,\beta)}=0\ \mbox{ in }
^{\displaystyle \C[\ku{Z},\ult]} \left/ 
_{\displaystyle \kI\cdot \C[\ku{Z},\ult]}\right.. \]
In particular, each relation $s\in\RI$ induces some element 
\[\lambda(s):=\sum\tilde{s}F\in\ 
W\otimes\C[\ku{z},t]\subseteq\ ^{\displaystyle \C[\ku{Z},\ult]} \left/ 
 _{\displaystyle \ktI\cdot\C[\ku{Z},\ult]}\right.\]
which does not depend on choices after the additional projection to $W\otimes_{\C}A(Y)$. 
This procedure describes a certain element 
$\lambda\in T_Y^2\otimes_{\C}W=\Hom(W^*,T_Y^2)$ called the
obstruction map.

The remaining part of Sect.\ \ref{obst} contains the proof of the
following theorem:

\begin{theorem} \label{obst-them}
The obstruction map $\lambda:W^*\to T_Y^2$ is injective.
\end{theorem}

\subsection{}\label{obst-3}
We have to improve our notation of Sects.\ \ref{tautco}
and \ref{flat}. Since 
$\CM\subseteq \bar S\subseteq\C^{\m}$,
we were able to use the toric equations, cf.\ (\ref{def-4}) during 
computations modulo $\kI$.
In particular, the exponents $\ku{\eta}\in V^*$ of $\ult$
needed only to be known modulo $V^{\bot}$; it was enough to define 
$\ues(\kbb)$ as elements of $V_{\Z}^*$.\\
However, to compute the obstruction map, we have to deal with the smaller
ideal $\ktI\subseteq \kI$. Let us start with refining the definitions
of (\ref{tautco-6}):
\\[1.7ex]
(i) For each vertex $v\in Q$, we
choose certain paths through the
1-skeleton of $Q$:
\vspace{-1.7ex}
\begin{itemize}
\item[$\bullet$]
$\ku{\lambda}(v):=$ path from $0\in Q$ to $v\in Q$.
\item[$\bullet$]
$\ku{\mu}^{\nu}(v):=$ path from $v\in Q$ to $v(c^{\nu}) \in Q$ such that
$\mu_i^{\nu}(v) \langle d^i , c^{\nu} \rangle \le 0$ for each $i=1,\dots,N$.
\item[$\bullet$]
$\ku{\lambda}^{\nu}(v):= \ku{\lambda}(v) + \ku{\mu}^{\nu}(v)$ is then a path
from $0\in Q$ to $v(c^{\nu})$ depending on $v$.
\end{itemize}

(ii)
For each $c\in (\Qinf)^{\veee}$, 
we use the vertex $v(c)$ to define
\[
\ku{\eta}^c(c):= \left[ -\lambda_1(v(c))\langle d^1,c\rangle,
\dots, -\lambda_N(v(c)) \langle d^N,c\rangle \right] \in \Q^N
\]
and
\[
\ku{\eta}^c(c^{\nu}):= \left[ -\lambda_1^{\nu}(v(c))\langle d^1,c^{\nu}\rangle,
\dots, -\lambda_N^{\nu}(v(c)) \langle d^N,c^{\nu}\rangle \right] \in \Q^N\,.
\]
Additionally, if $v(c)\notin \lat$, 
we need to define
\[
\uecs(c):=\ku{\eta}^c(c)+[\ezs(c)-\eta_0(c)]\cdot e[v(c)]\] and \[
\uecs(c^{\nu}):=\ku{\eta}^c(c^{\nu})+[\ezs(c^{\nu})-\eta_0(c^{\nu})]\cdot 
e[v(c^{\nu})].
\]
\\[0.5ex]
(iii)
For each $c\in (\Qinf)^{\veee}\cap\lat^*$ we fix a representation $c=\sum_{\nu} 
p_{\nu}^c\, c^{\nu}$
($p_{\nu}^c\in \N$) such that $[c,\ezs(c)]=\sum_{\nu} p_{\nu}^c [c^{\nu},\ezs(c^{\nu})]$. 
(That
means, $c$ is represented only by those generators $c^{\nu}$ that define
faces of $Q$ containing the face defined by $c$ itself.)
Now, we improve the definition of the polynomials $F_\bullet(\ku{Z},\ult)$
given in (\ref{flat-3}). 
Let $a,b\in \N^w, \alpha,\beta\in \N$ such that
$([a,\alpha],[b,\beta])\in m\subseteq\N^{w+1}\times \N^{w+1}$, i.e.\
\[
c:= \sum_{\nu} a_{\nu}\, c^{\nu} = \sum_{\nu} b_{\nu}\, c^{\nu}\quad \mbox{and}\quad
\sum_{\nu} a_{\nu}\, \ezs(c^{\nu}) + \alpha = \sum_{\nu} b_{\nu} \, \ezs(c^{\nu}) + \beta\,.
\vspace{-0.5ex}
\]
Then
\[
F_{(a,b,\alpha,\beta)}(\ku{Z},\ult) := f_{(a,b,\alpha,\beta)}(\ku{Z},t_1) -
\ku{Z}^{\ku{p}^c} \cdot
\left( \ult^{\alpha e_1 +\sum_{\nu} a_{\nu} \uecs(c^{\nu})-\uecs(c)} -
\ult^{\beta e_1 +\sum_{\nu} b_{\nu} \uecs(c^{\nu})-\uecs(c)} \right).%
\]

\subsection{}\label{obst-threeRel}
We need to discuss the same three types of relations as we did in
(\ref{flat-three-Rel}). 
Since there is only one single element $c\in \lat$ involved
in the relations (i) and (ii), computing modulo $\ktI$ instead of $\kI$
makes no difference in these cases --
we always obtain $\lambda(s)=0$.
Let us consider the third
relation $s:=\left[\ku{z}^r\cdot f_{(a,b,\alpha,\beta)} -
f_{(a+r,b+r,\alpha,\beta)} =0\right]$ $\;(r\in \N^w$).
We will use the following notation:
\vspace{-1ex}
\begin{itemize}
\item
$c:=\sum_{\nu} a_{\nu} \, c^{\nu} = \sum_{\nu} b_{\nu} \, c^{\nu};\quad
\ku{p}:=\ku{p}^c;\quad
\ues:= \uecs;$
\vspace{1ex}
\item
$\tilde{c}:=\sum_{\nu} (a_{\nu}+r_{\nu}) \, c^{\nu} = \sum_{\nu} (b_{\nu}+r_{\nu}) \, c^{\nu} =
\sum_{\nu} (p_{\nu}+r_{\nu}) \, c^{\nu}; \quad
\ku{q}:=\ku{p}^{\tilde{c}};\quad
\ku{\tilde{\eta}}^*:= \ku{\eta}^{*\tilde{c}};$%
\vspace{1ex}
\item
$\xi:= \sum_i\left(\left( \sum_{\nu} (p_{\nu}+r_{\nu})\tilde{\eta}^*_i(c^{\nu})\right)
-\tilde{\eta}^*_i(\tilde{c})\right) = \sum_{\nu} (p_{\nu}+r_{\nu})\ezs(c^{\nu}) -
\ezs(\tilde{c})\,.$
\end{itemize}
Using the same lifting of $s$ to $\tilde{s}$ as in 
(\ref{flat-three-Rel}) yields
\[
\begin{array}{rcl}
\lambda(s)&=&\!\begin{array}[t]{l}
\ku{Z}^r\cdot F_{(a,b,\alpha,\beta)}- F_{(a+r,b+r,\alpha,\beta)} \,-
\vspace{1ex}\\
\qquad\qquad -\,
\left( \ult^{ \alpha e_1 + \sum_{\nu}a_{\nu}\ues(c^{\nu})-\ues(c)} -
\ult^{ \beta e_1 + \sum_{\nu}b_{\nu}\ues(c^{\nu})-\ues(c)} \right)\cdot
F_{(q,p+r,\xi,0)}
\end{array}
\vspace{2ex}\\
&=&\!
\begin{array}[t]{l}
-\ku{Z}^{p+r}\cdot \left( \ult^{\alpha e_1 + \sum_{\nu}(a_{\nu}-p_{\nu})\ues(c^{\nu})}
- \ult^{\beta e_1 + \sum_{\nu}(b_{\nu}-p_{\nu})\ues(c^{\nu})} \right)\, +
\vspace{1ex}\\
+\,
\ku{Z}^q\cdot \left( \ult^{\alpha e_1 + \sum_{\nu}(a_{\nu}+r_{\nu}-q_{\nu})\ku{\tilde{\eta}}^*(c^{\nu})}
- \ult^{\beta e_1 + \sum_{\nu}(b_{\nu}+r_{\nu}-q_{\nu})\ku{\tilde{\eta}}^*(c^{\nu})} \right)\, -
\vspace{1ex}\\
-\,
\left( \ult^{\alpha e_1 + \sum_{\nu}(a_{\nu}-p_{\nu})\ues(c^{\nu})}
- \ult^{\beta e_1 + \sum_{\nu}(b_{\nu}-p_{\nu})\ues(c^{\nu})} \right) \cdot
\left(\ku{Z}^q\, \ult^{\sum_{\nu}(p_{\nu}+r_{\nu}-q_{\nu})\ku{\tilde{\eta}}^*(c^{\nu})}
-\ku{Z}^{p+r}\right)
\end{array}
\vspace{2ex}\\
&=&\!
\begin{array}[t]{l}
\ku{Z}^q\cdot
\left( \ult^{\alpha e_1 + \sum_{\nu}(a_{\nu}+r_{\nu}-q_{\nu})\ku{\tilde{\eta}}^*(c^{\nu})}
-\ult^{\alpha e_1 + \sum_{\nu}(p_{\nu}+r_{\nu}-q_{\nu})\ku{\tilde{\eta}}^*(c^{\nu}) +
\sum_{\nu} (a_{\nu}-p_{\nu})\ues(c^{\nu})}
\right) -
\vspace{1ex}\\
\qquad - \ku{Z}^q\cdot\left(
\ult^{\beta e_1 + \sum_{\nu}(b_{\nu}+r_{\nu}-q_{\nu})\ku{\tilde{\eta}}^*(c^{\nu})}
- \ult^{\beta e_1 + \sum_{\nu}(p_{\nu}+r_{\nu}-q_{\nu})\ku{\tilde{\eta}}^*(c^{\nu}) +
\sum_{\nu} (b_{\nu}-p_{\nu})\ues(c^{\nu})} \right).
\end{array}
\end{array}
\]
As in (\ref{flat-three-Rel})(iii), 
we can see that $\lambda(s)$ vanishes modulo $\kI$ 
(or even in $A(\bar{S})$) -- 
just identify $\ues$ and $\ku{\tilde{\eta}}^*$.

\subsection{}\label{obst-5}
In (\ref{obst-1}) we already mentioned the isomorphism
\[
W\otimes_{\C} \C[\ku{z},t] \stackrel{\sim}{\longrightarrow}\,
^{\kd \kI\cdot \C[\ku{Z},\ult]} \!\!\left/ \! 
_{\kd \ktI\cdot \C[\ku{Z},\ult]} \right.
\]
obtained by identifying $t$ with $t_1$ and $\ku{z}$ with $\ku{Z}$. Now, with 
$\lambda(s)$, we have obtained an element of the right hand side, which has to 
be interpreted as an element of $W\otimes_{\C} \C[\ku{z},t]$.
For this, we quote from
\cite[Lemma (7.5)]{gorenstein}:

\begin{lemma}
Let $A,B\in \N^N$ such that $\ku{d}:=A-B\in V^\bot$, 
i.e.\ $\ult^A-\ult^B \in
\kI\cdot \C[\ku{Z},\ult]$. Then, via the previously mentioned isomorphism,
$\ult^A-\ult^B$ corresponds to the element
\[
\sum_{k\geq 1} c_k\cdot g_{\ku{d},k}(\ult-t_1)\cdot t^{k_0-k}
\in W\otimes_{\C}
\C[\ku{z},t],
\]
where $k_0:=\sum_i A_i$, and
$c_k$ are some constants 
occurring in the context of symmetric polynomials,
cf.\ \cite[(3.4)]{gorenstein}.
In particular, the coefficients from $W_k$ vanish for $k>k_0$.
\end{lemma}

\begin{corX}
Transferred to $W\otimes_{\C}\C[\ku{z},t]$, the element $\lambda(s)$ equals
\[
\sum_{k\geq 1} c_k \cdot g_{\ku{d},k}(\ult-t_1)\cdot \ku{z}^q \cdot 
t^{k_0-k}\quad
\mbox{ with} \begin{array}[t]{rcl}
\ku{d} &:=& \sum_{\nu} (a_{\nu}-b_{\nu})\cdot \left(\ku{\tilde{\eta}}(c^{\nu})-\ku{\eta}(c^{\nu}) 
\right)\,,\\
k_0 &:=& \alpha + \sum_{\nu} (a_{\nu}+r_{\nu})\, \ezs(c^{\nu}) - \ezs(\tilde{c})\,.
\end{array}
\]
The coefficients vanish for $k>k_0$.
\vspace{-1ex}
\end{corX}

\begin{proof} 
Since the $e[v(c)]$-terms kill each other,
one can easily see, that
\[ \ku{d} = \sum_{\nu} (a_{\nu}-b_{\nu})\cdot \left(\ku{\tilde{\eta}}(c^{\nu})-\ku{\eta}(c^{\nu}) 
\right) = \sum_{\nu} (a_{\nu}-b_{\nu})\cdot \left(\ku{\tilde{\eta}}^*(c^{\nu})-\ues(c^{\nu}) \right). \]
We apply the previous lemma to both the $a$- and the $b$-summand
of the $\lambda(s)$-formula of
(\ref{obst-threeRel}). For the first one we obtain
\begin{eqnarray*}
\ku{d}^{(a)} &=& \!
\begin{array}[t]{l}
[ \alpha e_1 + \sum_{\nu}(a_{\nu} +r_{\nu} -q_{\nu})\, \ues(c^{\nu})] \, -
\vspace{0.5ex}\\
\qquad\qquad - \,[\alpha e_1 + \sum_{\nu}(p_{\nu}+r_{\nu}-q_{\nu}) \, \ues(c^{\nu})
+\sum_{\nu} (a_{\nu}-p_{\nu})\, \ues(c^{\nu}) ]
\end{array}\\
&=&
\sum_{\nu} (a_{\nu}-p_{\nu})\cdot \left(\ku{\tilde{\eta}}^*(c^{\nu})-\ues(c^{\nu}) \right)
\qquad \mbox{and}
\vspace{2ex}\\
k_0 &=&
\sum_i \left(\alpha e_1 + \sum_{\nu}(a_{\nu}+r_{\nu}-q_{\nu})\,
\ku{\tilde{\eta}}^*(c^{\nu})\right)_i\\
&=&\alpha + \sum_{\nu}(a_{\nu}+r_{\nu}-q_{\nu}) \,\ezs(c^{\nu})\,
=\,\alpha + \sum_{\nu} (a_{\nu}+r_{\nu})\, \ezs(c^{\nu}) - \ezs(\tilde{c})\,.
\end{eqnarray*}
$k_0$ has the same value for both the $a$- and $b$-summand,
and
\[
\begin{array}{rcl}
\ku{d}& = & \ku{d}^{(a)}-\ku{d}^{(b)}\\ &=&
\sum_{\nu} (a_{\nu}-p_{\nu})\cdot \left(\ku{\tilde{\eta}}^*(c^{\nu})-\ues(c^{\nu}) \right) -
\sum_{\nu} (b_{\nu}-p_{\nu})\cdot \left(\ku{\tilde{\eta}}^*(c^{\nu})-\ues(c^{\nu}) \right) \\
&=&
\sum_{\nu} (a_{\nu}-b_{\nu})\cdot \left(\ku{\tilde{\eta}}^*(c^{\nu})-\ues(c^{\nu}) \right)\,.
\end{array}
\vspace{-3ex}
\]
\end{proof}

\subsection{}\label{obst-6}
Now, we try to approach the obstruction map $\lambda$ from the opposite 
direction.
Using the description of $T^2_Y$ given in (\ref{obst-1}) 
we construct an element
of $T^2_Y\otimes_{\C}W$ that, afterwards, will 
turn out to equal $\lambda$.
\\[1ex]
For $\rho\in \Z^N$ induced from some path along the edges of $Q$, we will
denote
\[
\ku{d}(\rho,c):= [\langle \rho_1\,d^1,\,c\rangle, \dots,
\langle \rho_N\,d^N,\,c\rangle]\in \R^{\m}
\]
the vector showing the behavior of $c\in\lat^*$ passing each particular edge.
If $\rho$ governs the walk between two lattice vertices and is regarded modulo
$t_i-t_j$ (if $d^i,\ d^j$ contain a common non-lattice vertex), then 
$\ku{d}(\rho,c)$
is contained in $\Z^{\m}$. In particular, this property holds for closed paths.
In this case $\ku{d}(\rho,c)$
will be contained in
$V^\bot$.
\\
On the other hand, for each $k\geq 1$, we can use the $\ku{d}$'s from $V^\bot$ 
to get
elements $g_{\ku{d},k}(\ult-t_1)\in W_k$ generating this vector space.
Composing both procedures we obtain, for each closed path $\rho\in \Z^N$, a map
\[
\begin{array}{cccccl}
g^{(k)}(\rho,\bullet):& \A^*&\longrightarrow &V^\bot &\longrightarrow &W_k\\
&c&&\mapsto && g_{\ku{d}(\rho,c), k}(\ult-t_1)\,.
\end{array}
\]

\begin{lemma}
(1) Taking the sum over all compact 2-faces we get a surjective map
\[
\sum_{\keps<Q} g^{(k)}(\ku{\keps},\bullet): \oplus_{\keps<Q}
\,\A^*\otimes_{\R}{\C} \surj W_k\,.
\vspace{-1ex}
\]
(2) Let $c\in\lat^*$ be integral.
If $\rho^1, \rho^2\in\Z^N$ are two paths each connecting vertices
$v,w\in Q$ such that 
\vspace{-2ex}
\begin{itemize}
\item[$\bullet$]
$|\langle v,c\rangle - \langle w,c \rangle | \le k-1\;$ and
\item[$\bullet$]
$c$ is monotone along both paths, i.e.\
$\langle \rho^{1/2}_i \,d^i , \,c \rangle
\ge 0$
for $i=1,\dots,N$,
\end{itemize}
\vspace{-2ex}
then $\rho^1-\rho^2\in \Z^N$ will be a closed path yielding
$g^{(k)}(\rho^1-\rho^2, \,c)=0$ in $W_k$.
\vspace{-2ex}
\end{lemma}

\begin{proof}
The reason for (1) is the fact that the elements $\ku{d}(\keps,c)$ ($\keps
<Q\,$ compact 2-face; $c\in \lat^*$) and $e_i-e_j$
(for $d^i,d^j$ containing a common non-lattice vertex) generate $V^\bot$
as a vector space; 
since $t_i-t_j\in\kI_1$ the latter type yields zero
in $W_k$.
\\[0.5ex]
For the proof of (2), we consider $\ku{d}:=\ku{d}(\rho^1-\rho^2,\,c)$.
Since
$d_i =\langle \rho^1_i\,d^i ,\,c\rangle -
\langle \rho^2_i\,d^i ,\,c\rangle$
is the difference of two non-negative integers, we obtain
$d_i ^+ \leq \langle \rho_i^1\,d^i ,\, c\rangle$.
Hence,
\[
\sum_i d_i ^+ \le \sum_i \langle \rho^1_i\,d^i ,\,c\rangle =
\langle w,c\rangle - \langle v,c\rangle \le k-1\,,
\]
and we obtain $g_{\ku{d},k}(\ult-t_1)
\in \ktI$ by the following corollary.
\end{proof}
\begin{corX}
Let $k_0:=\sum\ku{d}(\rho_1-\rho_2,c)^+$. Then 
$g_{\ku{d}(\rho_1-\rho_2,c),k}(\ult-t_1)\in\ktI$ for $k> k_0$.
\end{corX}
\begin{proof}
Consider $\ku{d}\in\VQ^{\bot}\cap\Z^N$.
From \cite{al3} proposition (2.3)
we know that $g_{\ku{d},k}(\ult-t_1)$ can be written as a
$\C[t_i-t_j]$-linear combinations of
$g_{\ku{d},1}(\ult-t_1),\ldots,g_{\ku{d},k_0}(\ult-t_1)$ for $k> k_0$, where
$k_0=\sum_id_i^+$. \\
Now $\ku{d}:=\ku{d}(\rho_1-\rho_2,c)\in\VQ^{\bot}$
does not have to be contained in $\Z^N$.
But since the path $\rho_1-\rho_2$ is closed, $\ku{d}$ yields an integer as sum
on every component. Since 
$\ku{d}_{ij}:=[0,\ldots,0,1_i,0,\ldots,0,-1_j,0,\ldots,0]\in V^{\bot}$
for $d^i,\ d^j$
containing a common non-lattice vertex we are able to find some 
$\tilde{\ku{d}}\in V^{\bot}\cap \Z^N$ with $\sum_i\tilde{d}_i^+=\sum_id_i^+$
such that $g_{\tilde{\ku{d}},k}(\ult-t_1)=g_{\ku{d},k}(\ult-t_1)+
\sum_{ij}q_{ij}\cdot g_{\ku{d}_{ij},k}(\ult-t_1)$, with the usual
assumptions for $i,\, j$.
In particular, the $q_{ij}$ do not depend on $k$.
For the $g_{\tilde{\ku{d}},k}(\ult-t_1)$ the first assumptions apply, and we
obtain for $k> k_0$:
\[
g_{\tilde{\ku{d}},k}(\ult-t_1)=\sum_{n=1}^{k_0}a_n(\ult-t_1)\cdot
g_{\tilde{\ku{d}},n}(\ult-t_1)
\]
and
\[
g_{\ku{d},k}(\ult-t_1)=\sum_{n=1}^{k_0}a_n(\ult-t_1)\cdot
g_{\tilde{\ku{d}},n}(\ult-t_1)-\sum_{ij}q_{ij}
\cdot g_{\ku{d}_{ij},k}(\ult-t_1).
\]
Now we assume w.l.o.g.\ that $a_n(\ult-t_1)$ is homogenous
and has degree $k-n$. Hence, the first sum on the right
hand side is contained in $(t_i-t_j)_{ij}\cdot \kI\subseteq\ktI$. 
Now consider the second sum. We know
$g_{\ku{d}_{ij},1}(\ult-t_1)\in\kI_1$ and $\sum_r(d_{ij}^+)_r=1$. Thus $g_{\ku{d}_{ij},k}(\ult-t_1)=f(\ult-t_1)\cdot 
g_{\ku{d}_{ij},1}(\ult-t_1)$ and this is contained in
$\kI_1\cdot\C[t_i-t_j]\subseteq\ktI$.
\end{proof}

\subsection{}\label{obst-sets}
Recalling the sets $E_j^k$ from (\ref{obst-1}),
we can define the following linear maps:
\[
\begin{array}{cccl}
\psi_j^{(k)}:& L(E_j^{k}) & \longrightarrow & W_k\\
& q & \mapsto & \sum_{\nu} q_{\nu}\cdot
g^{(k)}\left(\ku{\lambda}(v^j )+\ku{\mu}^{\nu}(v^j ) - 
\ku{\lambda}(v(c^{\nu})),\,c^{\nu}\right)\,.
\end{array}
\]
(The $q$-coordinate corresponding to $R\in E_j^{k}$ is not used in the
definition of $\psi_j^{(k)}$.)
\vspace{1ex}

\begin{lemma}
Let $\langle v^j ,v^l  \rangle <Q$ be an edge of the
polyhedron $Q$. Then, on
$L(E_j^{k} \cap E_l^{k}) = L(E_j^{k}) \cap
L(E_l^{k})$, the
maps $\psi_j^{(k)}$ and  $\psi_l^{(k)}$ coincide.
In particular (cf.\ Theorem \ref{obst-1}), the $\psi_j^{(k)}$'s
induce a linear map
$\psi^{(k)}: T^2_Y(-kR)^{*}\to W_k$.
\end{lemma}
\begin{proof}
The proof is similar to the proof of Lemma (7.6) in \cite{gorenstein}.
\end{proof}

Now, both ends will meet and we obtain an explicit description of
the obstruction map:

\begin{proposition} \label{obst-7}
$\;\sum_{k\ge 1} c_k\,\psi^{(k)}$ equals $\lambda^*$,
the adjoint of the obstruction map.
\vspace{-2ex}
\end{proposition}

\begin{proof}
Using Theorem 3.5 of \cite{al3}, we can
find an element of 
$\Hom(^{\displaystyle \RI}\!/\!
_{\displaystyle {\RI}_0},\, W_k\otimes A(Y))$
representing $\psi^{(k)}\in T^2_Y\otimes W_k$ --
it sends relations of type (i), cf.\ (\ref{flat-three-Rel}),
to 0 and deals with 
relations of type (ii) and (iii) in the following way:
\[
[\ku{z}^r\, t^{\gamma}\cdot f_{(a,b,\alpha,\beta)} -
f_{(a+r,b+r,\alpha+\gamma,\beta+\gamma)}=0]  \mapsto
\psi_j^{(k)}(a-b)\cdot x^{
\sum_{\nu} (a_{\nu}+r_{\nu})[c^{\nu},\ezs(c^{\nu})]+(\alpha+\gamma-k)R}\,,
\]
if
\[
\langle (Q,1),\, \sum_{\nu} (a_{\nu}+r_{\nu})\, [c^{\nu},\ezs(c^{\nu})] + (\alpha + \gamma -k)
R\rangle \geq 0\,,
\]
and $j$ is such that
\[
\langle (v^j ,1),\, \sum_{\nu} a_{\nu} \,[c^{\nu},\ezs(c^{\nu})] +
(\alpha -k)R\rangle <0\,;
\]
otherwise the relation is sent to $0$ 
(in particular, if there is not any $j$
meeting the desired condition).
\\[1ex]
On $Q$, the linear forms $c:=\sum_{\nu}a_{\nu}\,c^{\nu}$ and
$\tilde{c}=\sum_{\nu}(a_{\nu}+r_{\nu})c^{\nu}$ admit their minimal values at the vertices
$v(c)$ and $v(\tilde{c})$, respectively. Hence, we can transform the previous 
formula into
\[
\begin{array}{l}
[\ku{z}^r\, t^{\gamma}\cdot f_{(a,b,\alpha,\beta)} -
f_{(a+r,b+r,\alpha+\gamma,\beta+\gamma)}=0]  \mapsto
\psi_{v(c)}^{(k)}(a-b)\cdot x^{\sum_{\nu} (a_{\nu}+r_{\nu})[c^{\nu},\ezs(c^{\nu})]+
(\alpha+\gamma-k)R}
\\[1.7ex]
\begin{array}[t]{ll}
\mbox{if } &
\begin{array}[t]{l}
\sum_{\nu}(a_{\nu}+r_{\nu})\ezs(c^{\nu})-\ezs(\tilde{c})+(\alpha+\gamma-k) =
\\[0.5ex]
\qquad =
\langle (v(\tilde{c}),1),\, \sum_{\nu} (a_{\nu}+r_{\nu})\, [c^{\nu},\ezs(c^{\nu})] +
(\alpha + \gamma -k)
R\rangle \ge 0\,,
\end{array}
\vspace{1ex}
\\ &
\begin{array}[t]{l}
\sum_{\nu}a_{\nu}\,\ezs(c^{\nu})-\ezs(c)+(\alpha-k) =\\[0.5ex]
\qquad =
\langle (v(c),1),\, \sum_{\nu} a_{\nu} \,[c^{\nu},\ezs(c^{\nu})] +
(\alpha -k)R\rangle <0
\end{array}
\end{array}
\end{array}
\]
and mapping onto 0 otherwise.
\\[1ex]
Adding the coboundary $h\in \Hom\,(\C[\ku{z},t]^m,\, W_k\otimes A(Y))$
\[
h_{(a,\alpha), (b,\beta)}:=
\left\{ \begin{array}{ll}
\psi^{(k)}_{v(c)}(a-b)\cdot x^{\sum_{\nu} a_{\nu} [c^{\nu},\ezs(c^{\nu})] +(\alpha -k)R} &
\mbox{for } \sum_{\nu} a_{\nu}\,\ezs(c^{\nu})-\ezs(c)+\alpha\ge k\,,\\
0 & \mbox{otherwise}
\end{array} \right.
\]
does not change the class in $T^2_Y(-kR)\otimes W_k$
(still representing $\psi^{(k)}$), but
improves the representative from
$\Hom(^{\kd \RI}\!/\!_{\kd {\RI}_0},\, W_k\otimes A(Y))$.
It still maps type-(i)-relations to 0, and moreover
\[
\begin{array}{l}
[\ku{z}^r\, t^{\gamma}\cdot f_{(a,b,\alpha,\beta)} -
f_{(a+r,b+r,\alpha+\gamma,\beta+\gamma)}=0]  \mapsto
\vspace{1ex}\\
\quad\mapsto \left\{ \begin{array}{ll}
\left(\psi_{v(c)}^{(k)}(a-b)-
\psi_{v(\tilde{c})}^{(k)}(a-b)\right)
\cdot x^{\sum_{\nu} (a_{\nu}+r_{\nu})[c^{\nu},\ezs(c^{\nu})]+
(\alpha+\gamma-k)R} &
\mbox{for }
k_0+\gamma\ge k\\
0 & \mbox{otherwise}\,
\end{array}
\right.
\end{array}
\]
with $k_0=\alpha + \sum_{\nu}(a_{\nu}+r_{\nu})\, 
\ezs(c^{\nu})-\ezs(\tilde{c})$.
By definition of $\psi^{(k)}_j$ and $g^{(k)}$ we obtain
\[
\begin{array}{l}
\psi_{v(c)}^{(k)}(a-b)- \psi_{v(\tilde{c})}^{(k)}(a-b)\, =
\vspace{0.5ex}\\
\qquad=\,
\sum_{\nu} (a_{\nu}-b_{\nu})\cdot g^{(k)}\left(
\ku{\lambda}(v(c)) + \ku{\mu}^{\nu}(v(c)) -
\ku{\lambda}(v(\tilde{c}) - \ku{\mu}^{\nu}(v(\tilde{c}),\;c^{\nu}\right)
\vspace{0.5ex}\\
\qquad=\,
\sum_{\nu} (a_{\nu}-b_{\nu})\cdot g^{(k)}\left(
\ku{\lambda}^{\nu}(v(c)) -
\ku{\lambda}^{\nu}(v(\tilde{c}) ,\,c^{\nu}\right)
\vspace{0.5ex}\\
\qquad=\,
g_{\ku{d},\,k}(\ult-t_1) \quad
\mbox{ with }
\begin{array}[t]{rcl}
\ku{d} &=&
\sum_{\nu}(a_{\nu}-b_{\nu})\cdot \ku{d} \left(
\ku{\lambda}^{\nu}(v(c)) -
\ku{\lambda}^{\nu}(v(\tilde{c}) ,\,c^{\nu}\right)\\
&=& \sum_{\nu} (a_{\nu}-b_{\nu})\cdot \left( \tilde{\ku{\eta}}(c^{\nu})-\ku{\eta}(c^{\nu})
\right)\\
&=& \sum_{\nu} (a_{\nu}-b_{\nu})\cdot \left( \ku{\tilde{\eta}}^*(c^{\nu})-\ues(c^{\nu})
\right)\,,
\end{array}
\end{array}
\]
and this completes our proof. Indeed,
for relations of type (ii)
(i.e. $r=0$; $\gamma=1$) we know $c=\tilde{c}$, hence, those relations map
onto 0.
For relations of type (iii) (i.e.\ $\gamma=0$) we can compare the previous
formula with the result obtained in Corollary~\ref{obst-5}:
The coefficients coincide, and
the monomial
$\ku{z}^q\,t^{k_0-k}\in \C[\ku{z},t]$ maps onto
$x^{\sum_{\nu} (a_{\nu}+r_{\nu})[c^{\nu},\ezs(c^{\nu})]+
(\alpha+\gamma-k)R}\in A(Y)$.
\vspace{-3ex}
\end{proof}

\subsection{}\label{obst-8}
It remains to show that the summands $\psi^{(k)}$ of $\lambda^*$ are
indeed surjective maps from $T^2_Y(-kR)^*$ to $W_k$. We will do
so by composing them with auxiliary surjective maps
$p^k: \oplus_{\keps<Q} \A^*\otimes_{\R}\C \surj T^2_Y(-kR)^*$
yielding
$\psi^{(k)}\circ p^k = \sum_{\keps<Q} g^{(k)}(\ku{\keps},\bullet)$. Then the 
result follows from the first part of Lemma \ref{obst-6}.
\\[1ex]
Let us fix some 2-face $\keps<Q$. Assume that $d^1,\dots,d^m$ are
its counterclockwise oriented edges, i.e.\ the sign vector $\ku{\keps}$ looks
like $\keps_i =1$ for $i=1,\dots,m$ and $\keps_j =0$ otherwise.
Moreover,
we denote the vertices of $\keps<Q$ by $v^1,\dots,v^m$ such that $d^i $
runs from $v^i $ to $v^{i+1}$ ($m+1:=1$).
\\[1ex]
Now $p^k$ maps $[c,z]\in M$ to the linear relation
\[
\sum_{i=1}^{m}\sum_{\nu}
(q_{i,\nu}-q_{i-1,\nu})
\cdot[c^{\nu},\ezs(c^{\nu})]+(q_i-q_{i-1})\cdot[\ku{0},1]=0,
\]
where
\[
[c,z]=\sum_{\nu}q_{i,\nu}[c^{\nu},\ezs(c^{\nu})]+q_i[\ku{0},1]
\]
with
$[c^{\nu},\ezs(c^{\nu})]\in E^k_i\cap E^k_{i+1}$ for every $q_{i,\nu}\not= 0$.
This relation is automatically contained in $\ker(\bigoplus_i L(E_i^k)\to
L(E))$.
Note that only the $c\in\lat^*$ is important; choosing another
$z$ will not change the differences $q_i-q_{i-1}$. A closer look at the
construction and the surjectivity can be taken in \cite{al3} sect. 6.
Finally, we apply $\psi^{(k)}$ to obtain
\[
\begin{array}{rcl}
\psi^{(k)}(p^k(c)) &=&
\sum_{i=1}^m \sum_{\nu} (q_{i,v}-q_{i-1,v})\cdot
g^{(k)} \left(
\ku{\lambda}(v^i )-\ku{\lambda}(v(c^{\nu})) + \ku{\mu}^{\nu}(v^i ),\;c^{\nu}\right)
\vspace{1ex}\\
&=& \! \begin{array}[t]{l}
\sum_{i,v} g^{(k)} \left(
\ku{\lambda}(v^i )-\ku{\lambda}(v(c^{\nu})) + \ku{\mu}^{\nu}(v^i ),\;
q_{i,v}\,c^{\nu}\right)\,-
\vspace{0.5ex}\\
\qquad -\,
\sum_{i,v} g^{(k)} \left(
\ku{\lambda}(v^{i+1})-\ku{\lambda}(v(c^{\nu})) + \ku{\mu}^{\nu}(v^{i+1}),\;
q_{i,v}\,c^{\nu}\right)
\vspace{1ex}
\end{array}\\
&=&
\sum_{i,v} g^{(k)} \left(
\ku{\lambda}(v^i ) -\ku{\lambda}(v^{i+1})
+ \ku{\mu}^{\nu}(v^i )- \ku{\mu}^{\nu}(v^{i+1}),\;
q_{i,v}\,c^{\nu}\right)\,.
\end{array}
\vspace{-1ex}%
\]

We introduce the path $\rho^i$
consisting of the single edge $d^i $ only. Then, if $q_{i v}\not= 0$
and w.l.o.g. $\langle v^i ,c^{\nu}\rangle \ge \langle v^{i+1},c^{\nu}\rangle$,
the pair of paths $\ku{\mu}^{\nu}(v^i )$ and $\ku{\mu}^{\nu}(v^{i+1})+\rho^i$
meets the assumption of Lemma \ref{obst-6}(2) (cf.\ (i)).
Hence, we can proceed as follows:
\[
\begin{array}{rcl}
\psi^{(k)}(p^k(c)) &=& \!
\begin{array}[t]{l}
\sum_{i,v} g^{(k)} \left(
\ku{\lambda}(v^i ) -\ku{\lambda}(v^{i+1}) + \rho^i,\,q_{i v}\,c^{\nu} \right)
\,+
\vspace{0.5ex}\\
\qquad\qquad\qquad +\,\sum_{i,v} g^{(k)} \left(
\ku{\mu}^{\nu}(v^i )- \ku{\mu}^{\nu}(v^{i+1}) - \rho^i ,\,q_{i v}\,c^{\nu} \right)
\end{array}
\vspace{1ex}\\
&=&
\sum_{i=1}^m g^{(k)}\left(
\ku{\lambda}(v^i ) -\ku{\lambda}(v^{i+1}) + \rho^i,\,
\sum_{\nu}q_{i v}\,c^{\nu} \right)
\vspace{1ex}\\
&=&
\sum_{i=1}^m g^{(k)}\left(
\ku{\lambda}(v^i ) -\ku{\lambda}(v^{i+1}) + \rho^i,\, c \right)
\vspace{1ex}\\
&=&
g^{(k)}\left( \sum_{i=1}^m \rho^i,\,c \right)
\vspace{1ex}\\
&=&
g^{(k)}(\ku{\keps},\,c) \,.
\vspace{0.1ex}
\end{array}
\]

Thus, Theorem \ref{obst-them} is proven.

\section{Example}
First let us provide a theorem to describe the situation for $\dim\sigma=3$.
We assume $\sigma$ is smooth in codimension two. Hence, it has an isolated
singularity and $\dim T_Y^1<\infty$, i.e.\ there are only finitely
many $R\in\sigv\cap M$ with $\dim(V(Q)/\ku{1})\not=0$. The second part of
the following theorem provides a combinatorial verification for this fact.
\begin{theorem} \label{3dim}
Let $\sigma\subset\R^3$ be a three dimensional cone
with smooth two dimensional faces.
\vspace{-1ex}
\begin{itemize}
\item[{\rm(i)}] 
Let $R\in\interior(\sigv\cap M)$. We define \[Q:=\sigma\cap [R=1] 
\mbox{ and }Q':=\conv(\mbox{\rm lattice vertices of }Q).\] Define
$\sigma':=\Cone{Q'}$. Then we denote by $Y':=\toric{\sigma'}$ the associated Gorenstein
singularity. If the edge vectors of $Q'$ are primitive (i.e. $\sigma'$ has
smooth two dimensional faces), then $Y'$ has the same deformation theory in
degree $R^*$ as $Y$ in degree $R$.
\vspace{1ex}
\item[\rm{(ii)}] 
There are only finitely many $R\in\sigv\cap M$ such that
$\dim(V(Q)/\ku{1})\not=0$.
\end{itemize}
\end{theorem}
\begin{proof}
\begin{itemize}
\item[(i)] This is obvious, since $V(Q)\cong V(Q')$.
\vspace{1ex}
\item[(ii)] Let $R\in\interior(\sigv\cap M)$. Then $Q:=\sigma\cap [R=1]$ is
a two dimensional polytope.
We know $\dim{T_Y^1(-R)}=\dim{V(Q)}-1$ by (\ref{ks-thm}), hence, to obtain
$\dim{T_Y^1(-R)}\ge 1$ we need $\dim{V(Q)}\ge 2$. Therefore, $Q$ has to
have at least four different components. This is equivalent to $Q$ having
at least four lattice vertices. Now for any four generating rays of sigma
there are at most one $R\in\interior(\sigv\cap M)$ yielding one on all four of
them.\\[1.5ex]
Let us now assume $R\in\partial(\sigv\cap M)$. 
If $Q$ has less than three vertices, we immediately obtain $\dim{V(Q)}\le 1$.
Otherwise $Q$ looks like:
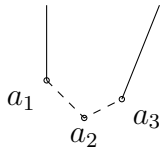
\[
\begin{tikzpicture}[scale=0.5]
\draw(0,3)--(0,1);
\draw[dashed] (0,1) -- (1,0) -- (2,0.5);
\draw(2,0.5)--(3,3);
\draw (0,1) circle (2pt) node[anchor=north east] {$a_1$};
\draw (1,0) circle (2pt) node[anchor=north] {$a_2$};
\draw (2,0.5) circle (2pt) node[anchor=north west] {$a_3$};
\end{tikzpicture}
\]
Assume $Q$ has at least two lattice vertices. Then $R$ yields $1$ on two rays of
$\sigma$. Since $R\in\partial(\sigv\cap M)$, $R$ yields zero on at least
one ray of $\sigma$. These conditions fully determine $R$ and since $\sigma$
is spanned by finitely many rays there are only finitely many such $R$.
\\[1.5ex]
Now assume $Q$ has only one lattice vertex. If this lattice vertex is
$a_1$ or $a_3$, we immediately obtain $\dim{V(Q)}=1$. To obtain
$\dim{V(Q)}>1$ the lattice vertex has to be $a_2$. Let $a_4$ be a ray of
$\sigma$ such that $R(a_4)=0$. By the above observation we know
that $a_2$ and $a_4$ do not lie in a common two face of $\sigma$. There
are only two facets of $\sigv$ that have an infinite intersection
with the hyperplane $[a_2=1]$, i.e. those defined by $a_1$ and $a_3$.
Hence, the set $[a_2=1]\cap[a_4=0]\cap\sigv$ is bounded and
$[a_2=1]\cap[a_4=0]\cap\sigv\cap M$ is finite.
\end{itemize}
\end{proof}

Finally we provide an example to illustrate the whole theory
and, in particular, Theorem \ref{3dim}(i).
Let $N=\Z^3$ be a lattice. Define the cone $\sigma$ by
\[
\sigma:=\langle 
(0,0,1),(1,0,1),(2,1,1),(1,2,1),(1,4,2),(0,1,2)\rangle\subseteq\Q^3=N_{\Q}.
\]
We choose $R:=[0,0,1]\in M=\Z^3$ and obtain the following polygon $Q$:
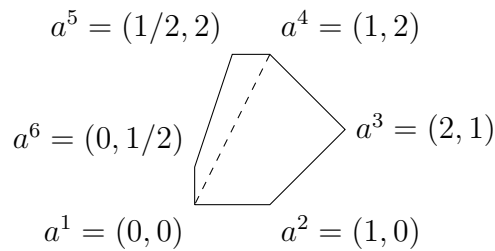
\begin{center}
\begin{tikzpicture}
\draw (0,0) node[anchor=north east] {$a^1=(0,0)$}
-- (1,0) node[anchor=north west] {$a^2=(1,0)$}
-- (2,1) node[anchor=west] {$a^3=(2,1)$}
-- (1,2) node[anchor=south west] {$a^4=(1,2)$}
-- (0.5,2) node[anchor=south east] {$a^5=(1/2,2)$}
-- (0,0.5) node[anchor=south east] {$a^6=(0,1/2)$}
-- (0,0);
\draw[dashed] (1,2) -- (0,0);
\end{tikzpicture}
\end{center}
We obtain the following paths:
\[
d^1=\twovec{1}{0},\ d^2=\twovec{1}{1},\ 
d^3=\twovec{-1}{1},
\]
\[ d^4=\twovec{-\frac{1}{2}}{0},\
d^5=\twovec{-\frac{1}{2}}{-\frac{3}{2}},\ d^6=\twovec{0}{-\frac{1}{2}}.
\]
Let $Q'$ be the convex hull of the lattice vertices of $Q$, consisting of $d^1$, $d^2$, $d^3$ and the dashed line in the picture. 
The associated Gorenstein singularity
$Y'=\toric{\sigma'}$ with $\sigma'$ being the cone over $Q'$
equals the affine cone over the Del Pezzo surface of degree $8$. 
It would be interesting to know more about a general geometric relation
between the singularities $Y$ and $Y'$, i.e.\ is there a universal property
(depending on $R$)
characterizing the map $Y'\to Y$?

Now we can explicitly describe the ideal $\kI$ as defined in 
(\ref{def-higherdeg}): $Q$ equals its own (and only)
2-face. This yields the following families of polynomials:
\[
g_{1,k}(\ult)=t_1^k+t_2^k-t_3^k-\frac{1}{2}t_4^k-\frac{1}{2}t_5^k,\, k\ge 1
\]
and
\[
g_{2,k}(\ult)=t_2^k+t_3^k-\frac{3}{2}t_5^k-\frac{1}{2}t_6^k,\, k\ge 1.
\]
Additionally, we have the polynomials $t_4-t_5$ and $t_5-t_6$ for the
non-lattice vertices. We obtain
\[
\kI=(g_{1,k}(\ult),g_{2,k}(\ult)|\, k\ge 1)+(t_4-t_5,t_5-t_6).
\]
By Corollary \ref{obst-6} we know that it is enough to consider
$k\le 3$. Calculating modulo the the two last equations and hence, only considering $t_4$, the homogeneous ideal $\kI$ defining $\CM\subseteq\C^4$ is generated by:
\[
\kI = (t_2+t_3-2\cdot t_4,t_1-2\cdot t_3+t_4,t_3^2-2\cdot t_3t_4+t_4^2).
\]
We introduce the variables $w_{12}:=t_1-t_2$, $w_{23}:=t_2-t_3$ and $w_{34}:=t_3-t_4$ for the differences $t_i-t_j$. Now one can easily see that
\[
\kI = (w_{23}+2\cdot w_{34}, w_{12}+w_{23}-w_{34}, w_{34}^2)
\]
holds as predicted by Theorem \ref{def-4}. Moreover these equations define $\bar{\CM}\subseteq\C^3$.
\par
Let us now construct $V(Q)$ as described in (\ref{def-eps}). Since $Q$ has $6$ edges we obtain a description
of $V(Q)$ as a subspace of $\R^6$.

The polygon $Q$ has two non-lattice vertices, namely $a^5$ and $a^6$. These
vertices are directly connected by edge $d^5$ and together they form a component
of $Q$, shown by the dashed line in the picture. This yields the equations
$t_4=t_5$ and $t_5=t_6$ for $\ult\in\VQ$. From now on we will calculate modulo these equations and hence, only consider $t_4$.

The remaining two equations are
described by the rows of $\sum t_id^i=0$.
Since all these equations are linearly independent we obtain that $\VQ$
is a two-dimensional subspace of $\R^6$.\par
The next step is to compute the Hilbert basis $E$ of $\sigv\cap M$. To do this,
we use a program like \cite{normaliz}:
\[
\begin{split}
E=\{ &
R=[0,0,1]
,[6,-2,1]
,[1,0,0]
,[0,1,0]
,[2,-1,1]
,\\&[-1,-1,3]
,[-1,1,1]
,[0,-1,2]
,[-1,0,2]\}
\end{split}
\]
Using the elements of $E\backslash\{R\}$, we can describe
$\tautv\cap\tilde{M}$. However, since we calculate modulo $t_4=t_5$ and $t_5=t_6$ as described above, we will not denote the $\ues(c^i)$ as elements
of $\R^6$, instead we will consider the evaluation of $\ues(c^i)$ on components
of $\ult\in V$, i.e. in the case of
$\eta^*_4(c^i),\ \eta^*_5(c^i),\ \eta^*_6(c^i)$ it is only important to know their sum. This means our $\ues(c^i)$ are built up by the following formula:
\[
\ues(c^i)=[\eta^*_1(c^i),\eta^*_2(c^i),\eta^*_3(c^i),
\eta^*_4(c^i)+\eta^*_5(c^i)+\eta^*_6(c^i)]\in\R^4.
\]
\[
\begin{tabular}{c|c|c|c|c}
$i $&$c^i $&$ v(c^i) $&$ \ku{\lambda}^{c^i} $&$ \ues(c^i)$\\
\midrule
1
&[6, -2] & (0, 1/2) & (0, 0, 0, 0, 0, -1) & [0, 0, 0, 1]\\
2
&[1, 0] & (0, 0) & (0, 0, 0, 0, 0, 0) & [0, 0, 0, 0]\\
3
&[0, 1] & (0, 0) & (0, 0, 0, 0, 0, 0) & [0, 0, 0, 0]\\
4
&[2, -1] & (1/2, 2) & (0, 0, 0, 0, -1, -1) & [0, 0, 0, 1]\\
5
&[-1, -1] & (2, 1) & (1, 1, 0, 0, 0, 0) & [1, 2, 0, 0]\\
6
&[-1, 1] & (1, 0) & (1, 0, 0, 0, 0, 0) & [1, 0, 0, 0]\\
7
&[0, -1] & (1/2, 2) & (0, 0, 0, 0, -1, -1) & [0, 0, 0, 2]\\
8
&[-1, 0] & (2, 1) & (1, 1, 0, 0, 0, 0) & [1, 1, 0, 0]
\end{tabular}
\]
Using the Hilbert basis $E$ of $\sigma^{\veee}\cap M$ we want to describe the affine toric variety $Y$ as a subvariety of $\C^{|E|}=\C^9$.
To do this, consider the following exact sequence:
\[
0\to L\to \Z^9\stackrel{\pi}{\to} M\to 0,
\]
where $\pi$ is defined by mapping the $e_i\in\Z^9$ to the generators of
the Hilbert basis of $\sigv\cap M$, i.e. the matrix
\[
\pi=\left(\begin{array}{ccccccccc}
0 & 6 & 1 & 0 & 2 & -1 & -1 & 0 & -1\\
0 & -2 & 0 & 1 & -1 & -1 & 1 & -1 & 0\\
1 & 1 & 0 & 0 & 1 & 3 & 1 & 2 & 2
\end{array}\right).
\]
Let $L$ be the kernel of this matrix. We build up the
so called toric ideal 
\[
I_L:=(\ku{x}^{l^+}-\ku{x}^{l^-}\,|\, l\in L)\, \subseteq\, k[x_0,\ldots,x_8]
\]
and obtain
\[
k[\sigv\cap M]\,\cong\,k[\pi(\N^n)]\,\cong\,
^{\displaystyle {k[\ku{x}]}}\!/_{\displaystyle {I_L}} .
\]
This yields an inclusion $Y=\Spec{(k[\sigv\cap M])}\subseteq \C^9$. Now we need
to compute the generators of the ideal $I_L$, which can be easily done by using
{\sc toric.lib} of \cite{singular}.
The following code will do the calculation needed:
\begin{verbatim}
LIB "toric.lib";
ring r=0,(t,z1,z2,z3,z4,z5,z6,z7,z8),dp;
intmat pi[3][9]=
0,6,1,0,2,-1,-1,0,-1,
0,-2,0,1,-1,-1,1,-1,0,
1,1,0,0,1,3,1,2,2;
pi;
ideal I=toric_ideal(pi,"pt");
def L=mstd(I);
I=L[2];
I;
\end{verbatim}
Note that we chose the variables of the ring according to the description
given in (\ref{flat-3}),
i.e. $t$ corresponds to $R\in M$ and $z_i$ corresponds to $c^i$.
We obtain the following polynomials defining $Y$:
\[
\begin{tabular}{clcl}
$0$ & $ f_{(e^{6}+e^{7},e^{8},0,1)}$&$=$&$z_6z_7-tz_8$\\ 
$1$ & $ f_{(e^{3}+e^{7},e^{2}+e^{8},0,0)}$&$=$&$z_3z_7-z_2z_8$\\ 
$2$ & $ f_{(e^{5}+e^{6},2 e^{8},0,0)}$&$=$&$z_5z_6-z_8^2$\\ 
$3$ & $ f_{(e^{6},e^{3}+e^{8},1,0)}$&$=$&$tz_6-z_3z_8$\\ 
$4$ & $ f_{(e^{3}+e^{5},e^{8},0,1)}$&$=$&$z_3z_5-tz_8$\\ 
$5$ & $ f_{(e^{2}+e^{5},e^{7},0,1)}$&$=$&$z_2z_5-tz_7$\\ 
$6$ & $ f_{(e^{5},e^{7}+e^{8},1,0)}$&$=$&$tz_5-z_7z_8$\\ 
$7$ & $ f_{(e^{3},e^{2}+e^{6},1,0)}$&$=$&$tz_3-z_2z_6$\\ 
$8$ & $ f_{(0,e^{2}+e^{8},2,0)}$&$=$&$t^2-z_2z_8$\\ 
$9$ & $ f_{(e^{2}+2 e^{7},e^{4}+e^{5},0,0)}$&$=$&$z_2z_7^2-z_4z_5$\\ 
$10$ & $ f_{(2 e^{2}+e^{8},e^{4}+e^{6},0,0)}$&$=$&$z_2^2z_8-z_4z_6$\\ 
$11$ & $ f_{(2 e^{2}+e^{7},e^{4},0,1)}$&$=$&$z_2^2z_7-tz_4$\\ 
$12$ & $ f_{(e^{2}+e^{7},e^{4}+e^{8},1,0)}$&$=$&$tz_2z_7-z_4z_8$\\ 
$13$ & $ f_{(3 e^{4},e^{1}+e^{7},0,0)}$&$=$&$z_4^3-z_1z_7$\\ 
$14$ & $ f_{(2 e^{2},e^{3}+e^{4},1,0)}$&$=$&$tz_2^2-z_3z_4$\\ 
$15$ & $ f_{(e^{2}+2 e^{4}+e^{7},e^{1}+e^{5},0,0)}$&$=$&$z_2z_4^2z_7-z_1z_5$\\ 
$16$ & $ f_{(e^{2}+e^{3}+2 e^{4},e^{1}+e^{6},0,0)}$&$=$&$z_2z_3z_4^2-z_1z_6$\\ 
$17$ & $ f_{(2 e^{2}+2 e^{4},e^{1},0,1)}$&$=$&$z_2^2z_4^2-tz_1$\\ 
$18$ & $ f_{(e^{2}+2 e^{4},e^{1}+e^{8},1,0)}$&$=$&$tz_2z_4^2-z_1z_8$\\ 
$19$ & $ f_{(4 e^{2}+e^{4},e^{1}+e^{3},0,0)}$&$=$&$z_2^4z_4-z_1z_3$
\end{tabular}
\]
We want to compute the liftings $F_{(a,b,\alpha,\beta)}$ of the $f_{(a,b,\alpha,\beta)}$ in $A(\bar{S})[Z_1,\ldots,Z_w]$. For a given
$c\in \lat^*$,
we have to find a representation 
$[c,\ues(c)]=\sum_{\nu}p_{\nu}[c^{\nu},\ues(c^{\nu})]$, $p_{\nu}\in\Z_{\ge 0}$.
This proves difficult, because we compute the $\ues$ modulo $\VQ^{\bot}$.
It is easier to use Proposition \ref{tautco-prop} instead. If we find a linear
combination $[c,\ezs(c)]= \sum_{\nu}p_{\nu}[c^{\nu},\ezs(c^{\nu})]$, we
automatically obtain $[c,\ues(c)]=\sum_{\nu}p_{\nu}[c^{\nu},\ues(c^{\nu})]$
with the same coefficients $p_{\nu}\in\Z_{\ge 0}$. Since $\sigv$
is a pointed cone and we already know a Hilbert basis of $\sigv\cap M$ this
problem is very easy to solve.\par
Using the equations $t_4=t_5$ and $t_5=t_6$ we obtain
\[
C(Q)^{\veee}= \;^{\displaystyle \Rpos^{4} + V^\bot}\!\!\left/
_{\displaystyle V^\bot} \right.
\]
where $V^\bot$ is generated by $[1,1,-1,-1]$ and $[0,1,1,-2]$ obtained from
the edge directions of the polygon $Q$. As introduced in (\ref{flat-3}) 
we will use this description of $C(Q)^{\veee}$ for the liftings of the $f_{(a,b,\alpha,\beta)}$, i.e. the variables $t_1,\ldots,t_4$ correspond to the coordinates of $\Rpos^4$. One can easily see that the exponents of the $t_i$ in an $F_{(a,b,\alpha,\beta)}$-term sum up to the exponent of $t$ in the corresponding term of $f_{(a,b,\alpha,\beta)}$. 
\[
\begin{tabular}{clcl}
$0$ & $ F_{(e^{6}+e^{7},e^{8},0,1)}$ 
& $=$ & $Z_{6}Z_{7}-Z_{8}t_1-Z_{8}(t_{3}-t_{1})$ \\ 
& & $=$ & $Z_{6}Z_{7}-Z_{8}t_3$ \\ 
$1$ & $ F_{(e^{3}+e^{7},e^{2}+e^{8},0,0)}$ 
& $=$ & $Z_{3}Z_{7}-Z_{2}Z_{8}-(t_{4}^{2}-t_{1}t_{2})$ \\ 
& & $=$ & $Z_{3}Z_{7}-t_{4}^{2}+F_8$ \\ 
$2$ & $ F_{(e^{5}+e^{6},2e^{8},0,0)}$ 
& $=$ & $Z_{5}Z_{6}-Z_{8}^{2}$ \\ 
$3$ & $ F_{(e^{6},e^{3}+e^{8},1,0)}$ 
& $=$ & $Z_{6}t_1-Z_{3}Z_{8}-Z_{6}(t_{1}-t_{2})$ \\ 
& & $=$ & $Z_{6}t_2-Z_{3}Z_{8}$ \\ 
$4$ & $ F_{(e^{3}+e^{5},e^{8},0,1)}$ 
& $=$ & $Z_{3}Z_{5}-Z_{8}t_1-Z_{8}(t_{2}-t_{1})$ \\ 
& & $=$ & $Z_{3}Z_{5}-Z_{8}t_2$ \\ 
$5$ & $ F_{(e^{2}+e^{5},e^{7},0,1)}$ 
& $=$ & $Z_{2}Z_{5}-Z_{7}t_1-Z_{7}(t_{4}-t_{1})$ \\ 
& & $=$ & $Z_{2}Z_{5}-Z_{7}t_4$ \\ 
$6$ & $ F_{(e^{5},e^{7}+e^{8},1,0)}$ 
& $=$ & $Z_{5}t_1-Z_{7}Z_{8}-Z_{5}(t_{1}-t_{3})$ \\ 
& & $=$ & $Z_{5}t_3-Z_{7}Z_{8}$ \\ 
$7$ & $ F_{(e^{3},e^{2}+e^{6},1,0)}$ 
& $=$ & $Z_{3}t_1-Z_{2}Z_{6}$ \\ 
$8$ & $ F_8:=F_{(0,e^{2}+e^{8},2,0)}$ 
& $=$ & $t_1^2-Z_{2}Z_{8}-(t_{1}^{2}-t_{1}t_{2})$ \\ 
& & $=$ & $t_1t_2-Z_{2}Z_{8}$ \\ 
$9$ & $ F_{(e^{2}+2e^{7},e^{4}+e^{5},0,0)}$ 
& $=$ & $Z_{2}Z_{7}^{2}-Z_{4}Z_{5}$ \\ 
$10$ & $ F_{(2e^{2}+e^{8},e^{4}+e^{6},0,0)}$ 
& $=$ & $Z_{2}^{2}Z_{8}-Z_{4}Z_{6}-Z_{2}(t_{1}t_{2}-t_{1}t_{4})$ \\ 
& & $=$ & $Z_{2}t_1t_4-Z_{4}Z_{6}-Z_{2}F_8$ \\ 
$11$ & $ F_{(2e^{2}+e^{7},e^{4},0,1)}$ 
& $=$ & $Z_{2}^{2}Z_{7}-Z_{4}t_1-Z_{4}(t_{4}-t_{1})$ \\ 
& & $=$ & $Z_{2}^{2}Z_{7}-Z_{4}t_4$ \\ 
$12$ & $ F_{(e^{2}+e^{7},e^{4}+e^{8},1,0)}$ 
& $=$ & $Z_{2}Z_{7}t_1-Z_{4}Z_{8}-Z_{2}Z_{7}(t_{1}-t_{3})$ \\ 
& & $=$ & $Z_{2}Z_{7}t_3-Z_{4}Z_{8}$ \\ 
$13$ & $ F_{(3e^{4},e^{1}+e^{7},0,0)}$ 
& $=$ & $Z_{4}^{3}-Z_{1}Z_{7}$ \\ 
$14$ & $ F_{14}:=F_{(2e^{2},e^{3}+e^{4},1,0)}$ 
& $=$ & $Z_{2}^{2}t_1-Z_{3}Z_{4}-Z_{2}^{2}(t_{1}-t_{4})$ \\ 
& & $=$ & $Z_{2}^{2}t_4-Z_{3}Z_{4}$ \\ 
$15$ & $ F_{(e^{2}+2e^{4}+e^{7},e^{1}+e^{5},0,0)}$ 
& $=$ & $Z_{2}Z_{4}^{2}Z_{7}-Z_{1}Z_{5}$ \\ 
$16$ & $ F_{(e^{2}+e^{3}+2e^{4},e^{1}+e^{6},0,0)}$ 
& $=$ & $Z_{2}Z_{3}Z_{4}^{2}-Z_{1}Z_{6}-Z_{2}^{3}Z_{4}(t_{4}-t_{1})$ \\ 
& & $=$ & $Z_{2}^3Z_{4}t_1-Z_{1}Z_{6}-Z_{2}Z_{4}F_{14}$ \\ 
$17$ & $ F_{(2e^{2}+2e^{4},e^{1},0,1)}$ 
& $=$ & $Z_{2}^{2}Z_{4}^{2}-Z_{1}t_1-Z_{1}(t_{4}-t_{1})$ \\ 
& & $=$ & $Z_{2}^{2}Z_{4}^{2}-Z_{1}t_4$ \\ 
$18$ & $ F_{(e^{2}+2e^{4},e^{1}+e^{8},1,0)}$ 
& $=$ & $Z_{2}Z_{4}^{2}t_1-Z_{1}Z_{8}-Z_{2}Z_{4}^{2}(t_{1}-t_{3})$ \\ 
& & $=$ & $Z_{2}Z_{4}^{2}t_3-Z_{1}Z_{8}$ \\ 
$19$ & $ F_{(4e^{2}+e^{4},e^{1}+e^{3},0,0)}$ 
& $=$ & $Z_{2}^{4}Z_{4}-Z_{1}Z_{3}$. 
\end{tabular}
\]
After reformulating the equations one easily notes that the ideal is indeed toric. To achieve positive exponents
in the $t_i$ it was necessary to compute modulo $V^{\bot}$.
These liftings together with the equations of $\kI$ describe a family contained
in \[\C^9\times\C^4\
\stackrel{\mbox{pr}_2}{\longrightarrow}\ ^{\displaystyle {\C^4}} \!/
_{\displaystyle {\C\cdot (\ku{1})}}.\]

\end{document}